\documentclass[reqno]{amsart}
\usepackage{amssymb}
\usepackage[usenames, dvipsnames]{color}

\usepackage{hyperref}

\usepackage{esint}
\usepackage{marginnote}
\usepackage{todonotes}
\usepackage{verbatim}

\theoremstyle{plain}
\newtheorem{theorem}{Theorem}[section]
\newtheorem{lemma}[theorem]{Lemma}
\newtheorem{corollary}[theorem]{Corollary}
\newtheorem{proposition}[theorem]{Proposition}

\theoremstyle{definition}
\newtheorem{definition}[theorem]{Definition}

\newtheorem{assumption}[theorem]{Assumption}

\theoremstyle{remark}
\newtheorem{remark}[theorem]{Remark}

\numberwithin{equation}{section}

\newcommand{\bR}{\mathbb{R}}
\newcommand{\bH}{\mathbb{H}}

\newcommand\cD{\mathcal{D}}

\newcommand\cH{\mathcal{H}}

\newcommand\cU{\mathcal{U}}

\providecommand{\norm}[1]{\lVert#1\rVert}

\makeatletter
\def\dashint{\operatorname%
{\,\,\text{\bf-}\kern-.98em\DOTSI\intop\ilimits@\!\!}}
\makeatother

\renewcommand{\epsilon}{\varepsilon}

\makeatletter
\@namedef{subjclassname@2020}{%
  \textup{2020} Mathematics Subject Classification}
\makeatother

\begin{document}

\title[Regularity estimates, singular-degenerate coefficients]{On parabolic and elliptic equations with singular or degenerate coefficients}

\author[H. Dong]{Hongjie Dong}
\address[H. Dong]{Division of Applied Mathematics, Brown University,
182 George Street, Providence, RI 02912, USA}
\email{Hongjie\_Dong@brown.edu}

\author[T. Phan]{Tuoc Phan}
\address[T. Phan]{Department of Mathematics, University of Tennessee, 227 Ayres Hall,
1403 Circle Drive, Knoxville, TN 37996-1320, USA}
\email{phan@math.utk.edu}

\thanks{T. Phan is partially supported by the Simons Foundation, grant \# 354889.}

%

\begin{abstract}
We study both divergence and non-divergence form parabolic and elliptic equations in the half space $\{x_d>0\}$ whose coefficients are the product of $x_d^\alpha$ and uniformly nondegenerate bounded measurable matrix-valued functions, where $\alpha \in (-1, \infty)$. As such, the coefficients are singular or degenerate near the boundary of the half space.  For equations with the conormal or Neumann boundary condition, we prove the existence, uniqueness, and regularity of solutions in weighted Sobolev spaces and mixed-norm weighted Sobolev spaces when the coefficients are only measurable in the $x_d$ direction and have small mean oscillation in the other directions in small cylinders. Our results are new even in the special case when the coefficients are constants, and they are reduced to the classical results when $\alpha =0$. \end{abstract}

\maketitle
\section{Introduction and main results}

In this paper, we study the existence, uniqueness, and regularity estimates of solutions in Sobolev spaces to a class of parabolic (and elliptic) equations in the upper half space, whose coefficients can be singular or degenerate on the boundary of the upper half space in a way which may not satisfy the classical Muckenhoupt $A_2$ condition.

Throughout the paper, let  $\Omega_T =(-\infty,T)\times \bR^d_+$ be a space-time domain, where $T\in (-\infty,+\infty]$, $\bR_+= (0, \infty)$,  and $\bR^d_+ = \bR^{d-1} \times \bR_+$ is the upper-half space. Let $(a_{ij}): \Omega_T \rightarrow \bR^{d\times d}$ be a matrix of measurable coefficients, which satisfies the following ellipticity and boundedness conditions: there is a constant $\kappa \in (0,1)$ such that
\begin{equation} \label{ellipticity}
\begin{split}
\kappa |\xi|^ 2 \leq a_{ij}(t,x) \xi_i\xi_j \quad \text{and} \quad |a_{ij}(t,x)| \leq \kappa^{-1}
\end{split}
\end{equation}
for every $\xi = (\xi_1, \xi_2,\ldots, \xi_d) \in \bR^d$ and $(t,x)=(t,x',x_d) \in \Omega_T$. Here we do not impose the symmetry condition on $(a_{ij})$.  Let $\alpha \in (-1, \infty)$ be a fixed number. We investigate the conormal boundary value problem
\begin{equation} \label{eq3.23}
\left\{
\begin{aligned}
x_d^\alpha ( u_t + \lambda  u ) - D_i[x_d^\alpha (a_{ij}(t,x) D_j u - F_i)] & = \sqrt{\lambda} x_d^\alpha f\\
\lim_{x_d \rightarrow 0^+}x_d^\alpha (a_{dj}(t,x) D_j u - F_d)  &= 0
\end{aligned} \right. \quad \text{in} \  \Omega_T,
\end{equation}
where $F = (F_1, F_2,\ldots, F_d): \Omega_T \rightarrow \bR^d$ and $f: \Omega_T \rightarrow \bR$ are given measurable functions in suitable weighted Lebesgue spaces, and $\lambda \geq 0$ is a parameter. It is worth noting that the weight $x_d^\alpha$ satisfies the Muckenhoupt $A_2$ condition only if $\alpha\in (-1, 1)$.  As a special case of our main results, for the model equation
\begin{equation}
                                \label{eq10.21}
\left\{\begin{aligned}
x_d^\alpha u_t-\textup{div}[x_d^\alpha (\nabla u -F)] &= x_d^\alpha f\\
\lim_{x_d \rightarrow 0^+} x_d^\alpha (D_d u-F_d)   & = 0
\end{aligned}\right.
\end{equation}
in the upper-half parabolic cylinder $Q_2^+$ and for $\alpha \in (-1, \infty)$, we obtain the local boundary weighted estimate
\begin{align*}
\Big(\int_{Q_1^+} [|u|^p+ |D u|^p] x_d^\alpha \, dz\Big)^{1/p}
&\leq N \int_{Q_2^+} [|u|+|Du|] x_d^\alpha \, dz \\
& + N \Big(\int_{Q^+_2} |F|^p x_d^\alpha \, dz\Big)^{1/p} +  N\Big(\int_{Q^+_2} |f|^{p^*} x_d^\alpha \, dz\Big)^{1/p^*}
\end{align*}
for every $p\in (1,\infty)$, where $p^* \in [1, p)$ depending on $\alpha$, $p$, and $d$  as in \eqref{eq3.19}-\eqref{eq3-2.19} below and $N>0$ is a constant depending on $d$, $\alpha, p$, and $p^*$. Equation \eqref{eq10.21} is related to the extension problem of the fractional heat operator (see, for example, \cite{NS16, ST17, BG}) and our result in this special case is already new.

We also consider the parabolic equation in non-divergence form
 \begin{equation} \label{non-div.eqn}
 a_0(t,x) u_t  - a_{ij}(t,x) D_{ij} u(t,x) - \frac{\alpha}{x_d} a_{dj}(t,x) D_j u(t,x) + \lambda c_0(t,x) u   = f \\
\end{equation}
in $\Omega_T$ with the boundary condition
\begin{equation} \label{non-div.bry}
\lim_{x_d \rightarrow 0^+} x_d^\alpha a_{dj}(t,x) D_j u(t,x', x_d) =0,
\end{equation}
where  $a_0, c_0: \Omega_T \rightarrow \mathbb{R}$ are measurable functions satisfying
\begin{equation} \label{a-c-eq}
\kappa \leq a_0(t,x), \ c_0(t,x) \leq \kappa^{-1}, \quad (t,x) \in \Omega_T.
\end{equation}
In this case, we impose an additional structural condition on the leading coefficients $a_{ij}$:
\begin{equation} \label{extension-type-matrix}
a_{dj}(t,x) =0, \quad j = 1,2,\ldots, d-1,
\end{equation}
or $a_{dj}=\lambda_{j}a_{dd}$ for $j=1,2,\ldots,d-1$ with constants $\lambda_{j}$, which can be reduced to \eqref{extension-type-matrix} after the change of variables $y_j = x_j - \lambda_j x_d$ for $j=1,2,\ldots, d-1$ and $y_d= x_d$. We note that this condition is satisfied for a large class of equations. See, for instance, \cite{BG, Cal-Syl, Pop-1, Gushin}.
Unlike \eqref{eq3.23}, the equation \eqref{non-div.eqn} has extra coefficients $a_0$ and $c_0$.  The main reason we introduce them in \eqref{non-div.eqn} is for convenience because in the proofs of main results for \eqref{non-div.eqn}-\eqref{non-div.bry}, we divide both sides of \eqref{non-div.eqn} by $a_{dd}$ to use the hidden divergence structure of   the equation.  Nevertheless, with $a_0$ and $c_0$ the equation \eqref{non-div.eqn} is slightly more general. Of course, in view of \eqref{a-c-eq}, by dividing both sides of \eqref{non-div.eqn} by $a_0$ or $c_0$, one can always assume one of them to be the identity.

The interest of studying the class of equations \eqref{eq3.23} and \eqref{non-div.eqn} comes from both pure mathematics and applied problems. As examples, we refer the reader to \cite{BG, Cal-Syl} for problems about fractional heat and fractional Laplace equations, and  \cite{Pop-1} for problems arising in mathematical finance. This paper is a continuation of \cite{Dong-Phan} and \cite{D-P}. In \cite{Dong-Phan}, we considered a class of parabolic equations in divergence form with a general weight
\begin{equation} \label{eq1.58}
 a_0(x_d)u_t - \frac 1 {\mu(x_d)} D_i[\mu(x_d) (a_{ij} D_j u - F_i)] +\lambda   u = f
\end{equation}
in the half space $\{x_d>0\}$ with conormal boundary condition:
\begin{equation} \label{main-eqn}
\lim_{x_d \rightarrow 0^+}\mu(x_d)(a_{dj}D_j u - F_d)  = 0.
\end{equation}
Here $(a_{ij})$ satisfies \eqref{ellipticity},
$a_0\in [\kappa,\kappa^{-1}]$, $\lambda\ge 0$, and the weight $\mu$ satisfies the $A_2$ condition and a relaxed $A_1$ type condition away from the boundary. This, in particular, includes the $A_2$ weights $\mu(x_d)=x_d^\alpha$ for any $\alpha\in (-1,1)$. We obtained the local and global weighted Calder\'on-Zygmund type estimates for \eqref{eq1.58}-\eqref{main-eqn} with respect to the weight $\mu$, under the condition that the coefficients are only measurable in the $x_d$ direction and have small mean oscillation in the other directions in small cylinders (partially VMO) with respect to the considered weight. The proofs in \cite{Dong-Phan} carry over to systems under the usual strong ellipticity condition. In \cite{D-P}, we studied the corresponding non-divergence form scalar equations \eqref{non-div.eqn}, where $\alpha\in (-1,1)$ and $a_0, c_0$ satisfy \eqref{a-c-eq}. Under the condition that $a_0$, $a_{ij}$, and $c_0$ are partially VMO respect to the weight $x_d^\alpha$, we obtained weighted mixed-norm $W^{1,2}_p$ estimates and solvability. The aim of this paper is to extend the results in \cite{Dong-Phan, D-P} to the full range of exponent $\alpha\in (-1,\infty)$. It is worth noting that even for divergence form equations, in contrast to \cite{Dong-Phan}, the proofs below only work for scalar equations because the Moser iteration is used (cf. Lemma \ref{lem2}). For other related work in this direction, we refer the reader to the references in \cite{Dong-Phan, D-P}.

The class of partially VMO coefficients was first introduced by Kim and Krylov \cite{MR2338417,MR2300337} for non-degenerate elliptic and parabolic equations in non-divergence form. Divergence form elliptic and parabolic equations with non-degenerate partially VMO coefficients were later studied in \cite{MR2601069, MR2584743}.  This type of equations arises from the problems of linearly elastic laminates and composite materials, e.g., in homogenization of layered materials. See, for instance, \cite{Chipot}. We also refer the reader to \cite{DK11,DK11b, MR3812104} for extensions to second-order and higher-order systems with or without weights.

We apply a mean oscillation argument, which was used in \cite{Krylov} for non-degenerate parabolic equations with coefficients which are VMO in the space variables. In the case of partially VMO coefficients, the main difficulty is that, since they are merely measurable in $x_d$, it is only possible to estimate the mean oscillation of $D_{x'}u$, not the full gradient $Du$. Therefore, one needs to bound $D_d u$ by $D_{x'}u$. An idea in \cite{MR2601069, MR2584743} is to break the ``symmetry'' of the coordinates so that $t$ and $x_d$ are distinguished from $x'$ by using a delicate re-scaling argument. Another idea is to estimate the mean oscillation of $a_{dd}D_d u$ instead of $D_d u$, and apply a generalized Fefferman--Stein theorem established in \cite{MR2540989}. In \cite{DK11}, a new method was developed, in which the key step is to estimate $\cU:=a_{dj} D_ju$ and $D_iu$, $i = 1, \ldots,d-1,$ instead of the full gradient of $u$. By using this argument,  one was able to bypass the scaling argument mentioned above and greatly simplified the proof. In this paper, we adapt this method to singular/degenerate equations.

In our main results, Theorems \ref{thm2}, \ref{thm1.4}, and \ref{para.theorem} 
below, we obtain the unique solvability \eqref{eq3.23} and \eqref{non-div.eqn}-\eqref{non-div.bry} in weighted Sobolev spaces and mixed-norm weighted Sobolev spaces.  Local boundary estimates for  solutions of these equations are also obtained in Corollaries \ref{main-thrm} and \ref{cor1.8}. To the best of our knowledge, these results are new even in the elliptic case and in the unmixed-norm case with constant coefficients $a_{ij}$, $a_0$, and $c_0$.

The proofs of the main theorems are based on an idea in \cite{DK11} mentioned above and the perturbation technique. To implement the method, we  first consider equations whose coefficients depend only on $x_d$ and prove various results on the existence, uniqueness, and regularity of solutions to this class of equations.  For this, we establish the $L_\infty$ estimate of weak solutions by applying the Moser iteration, and then derive Lipschitz and Schauder type estimates.  In particular, to estimate the $L_\infty$ norms of $D_d u$ and $\cU$, we use a bootstrap argument.  Schauder type estimates for elliptic equations similar to \eqref{eq.0612} were proved recently in \cite{Sire-1} when the matrix $(a_{ij})$ is symmetric, H\"older in all variables, and satisfies a structural condition that the hyperplane $\{x_d =0\}$ is  invariant with respect to $(a_{ij})$, i.e, $a_{jd}=d_{dj}=0$ for $j=1,\ldots,d-1$. The proof in \cite{Sire-1} uses a Liouville type theorem and a compactness argument. Our proof in Section \ref{div-sim-sec} is more direct and works for more general operators. For the local estimates Corollaries \ref{main-thrm} and \ref{cor1.8}, we prove a parabolic embedding (see Lemma \ref{lem2.2}) by using a generalized Hardy--Littlewood Sobolev inequality in \cite{HL28}, which seems to be new in the weighted setting and is of independent interest.  We also remark that in contrast to the previous work such as \cite{FKS, CPM, Dong-Phan} in which the $A_2$ weights are commonly assumed as the weighted Poincar\'e inequality is needed, we do not use the weighted Poincar\'e inequality in the proof.  In fact,  as pointed out in \cite{Sire-1}, when $\alpha\ge 1$, such inequality is not valid.

For simplicity, in this paper we choose not to consider lower-order terms. The results still hold for equations
$$
x_d^\alpha ( u_t -b_iD_i u-cu+ \lambda  u ) - D_i[x_d^\alpha (a_{ij} D_j u+\hat b_i u- F_i)] = \sqrt{\lambda} x_d^\alpha f
$$
and
$$
 a_0 u_t  - a_{ij} D_{ij} u - \Big(\frac{\alpha}{x_d} a_{dj}+b_i\Big) D_j u(t,x)-cu + \lambda c_0 u   = f,
$$
where $b_i$, $\hat b_i$, and $c$ are bounded measurable functions. To see this, it suffices to move the terms $b_iD_iu$ and $cu$ to the right-hand side of the equations, absorb $\hat b_i u$ to $F_i$, and take a sufficiently large $\lambda$. See, for example, \cite{Krylov} for details. By using the weighted embedding results such as Lemma \ref{lem2.2} below, it is also possible to consider unbounded lower-order coefficients. We refer the reader to the recent interesting work \cite{KK19,KT20, Kr20b,KRW20} and the references therein.

The remaining part of the paper is organized as follows. In the next section, we introduce some notation and state the main results of the paper.
In Section \ref{inq-sec}, we  prove two weighted embedding results that are needed in the paper as well as a result on the existence and uniqueness of $L_2$-solutions.  In Section \ref{div-sim-sec}, we study equations whose coefficients depend only on $x_d$. We prove the existence, uniqueness, and regularity estimates of solutions in $\cH_p^1(\Omega_T, \mu)$ after we obtain the $L_\infty$, Lipschitz, and Schauder type estimates for solutions to homogeneous equations. Finally, in Section \ref{mea-sec}, we provide the proofs of Theorems \ref{thm2} and \ref{thm1.4}, and Corollaries \ref{main-thrm} and \ref{cor1.8}.
\section{Notation and main theorems}
\subsection{Notation}
For $r>0$, $z_0 = (t_0, x_0)$ with $x_0 = (x_0', x_{0d}) \in  \bR^{d-1} \times \bR$ and $t_0 \in \bR$, we define $B_r(x_0)$  to be the ball in $\bR^d$ of radius $r$ centered at $x_0$, $Q_r(z_0)$ to be the parabolic cylinder of radius $r$ centered at $z_0$:
\[
Q_r(z_0) = (t_0-r^2 , t_0)\times B_r(x_0),
\]
and $B_r^+(x_0)$ and $Q_r^+(z_0)$ to be the upper-half ball and cylinder of radius $r$ centered at $x_0$ and $z_0$, respectively
\begin{align*}
B_r^+(x_0) &= \{x = (x_d, x') \in \bR^d:\, x_d >0, \ |x -x_0| < r\},\\
Q_r^+ (z_0)&= ( t_0-r^2, t_0)\times B_r^+ (x_0).
\end{align*}
When $x_0 =0$ and $t_0 =0$, for the simplicity of notation we drop $x_0, z_0$ and write $B_r$, $B_r^+$, $Q_r$, and $Q_r^+$, etc. 
We also define $B'(x_0')$ and $Q'(z'_0)$ to be the ball and the  parabolic cylinder in $\bR^{d-1}$ and $\bR^{d}$, where $z_0' = (t_0, x_0')$.

For $p \in (1, \infty)$,  $-\infty\le S<T\le +\infty$, and $\cD \subset \bR^d_+$, let $L_p((S,T)\times \cD, \mu)$ be the weighted Lebesgue space consisting of measurable function $g$ on $(S,T)\times \cD$ such that its norm
\[
\|g\|_{L_p( (S,T)\times \cD, \mu)} = \left( \int_{(S,T)\times \cD} |g(t,x)|^p\,\mu(dz) \right)^{1/p} <\infty,
\]
where $\mu(dz) = x_d^\alpha\,dxdt$. For $p,q\in (1,\infty)$ and the weights $\omega_0=\omega_0(t)$ and $\omega_1=\omega_1(x)$, we define $L_{q,p}(\Omega_T,\omega\,d\mu)$ to be the weighted mixed-norm Lebesgue space on $\Omega_T$ equipped with the norm
\begin{equation*}
\|f\|_{L_{q,p}(\Omega_T, \omega\, d\mu)}=\left(\int_0^T\Big(\int_{\bR^d_+} |f(t,x)|^p \omega_1(x)\,\mu(dx)\Big)^{q/p}\omega_0(t)\,dt\right)^{1/q},
\end{equation*}
where $\omega(t,x)=\omega_0(t)\omega_1(x)$. We also define
\[
\begin{split}
& \bH_{q, p}^{-1}( (S,T)\times \cD, \omega d\mu) \\
& =\{g:\, g  =  D_iF_i+F_0/x_d +f\ \ \text{for some}\ f\in L_{q, p}( (S,T)\times \cD, \omega d\mu)\\
& \qquad \ F= (F_0,\ldots,F_d) \in L_{q,p}((S,T)\times \cD, \omega d\mu)^{d +1}\}
\end{split}
\]
and
\[
\begin{split}
& \cH_{q,p}^1((S,T)\times \cD, \omega d\mu)\\
& =\{g\, : g,Dg\in L_p((S,T)\times \cD, \omega d\mu),\, g_t\in  \bH_{q,p}^{-1}( (S,T)\times \cD, \omega d\mu)\},
\end{split}
\]
which are equipped with the norms
\begin{align*}
\|g\|_{\bH_{q,p}^{-1}((S,T)\times \cD, \omega d\mu)} &=\inf\{\|F\|_{L_{q,p}((S,T)\times \cD, \omega d\mu)}
+\|f\|_{L_{q,p}((S,T)\times \cD, \omega d\mu)}: \\
& \quad \qquad \ g= D_iF_i+F_0/x_d+f\}
\end{align*}
and
\begin{align*}
&\|g\|_{\cH_{q,p}^1((S,T)\times \cD, \omega d\mu)} \\
&= \|g\|_{L_{q,p}((S,T)\times \cD, \omega d\mu)}
+ \|Dg\|_{L_{q, p}((S,T)\times \cD,\omega d \mu)}+\|g_t\|_{\bH_{q, p}^{-1}((S,T)\times \cD, \omega d\mu)}.
\end{align*}
When $p =q$, we simply write $\cH_p^1(\Omega_T, \omega d\mu) = \cH_{p,p}^1(\Omega_T, \omega d\mu)$. Similar notation are also used for other spaces. When $\omega\equiv 1$, we have $L_{q,p}(\Omega_T, \omega d\mu)  = L_{q,p}(\Omega_T, \mu)$, and similarly for other function spaces.

We say that $u \in \cH_{q,p}^1((S,T)\times \cD, \omega d\mu)$ is a weak solution of \eqref{eq3.23} in $(S,T)\times \cD$ if
\begin{align*}
&\int_{(S,T)\times \cD} (-u \partial_t \varphi +\lambda u \varphi)\,\mu(dz) + \int_{(S,T)\times \cD}( a_{ij} D_{j} u  - F_i)D_{i} \varphi\,\mu(dz) \\
&= \lambda^{1/2} \int_{(S,T)\times \cD} f(z) \varphi(z)\,\mu(dz)
\end{align*}
for any $\varphi \in C_0^\infty((S,T)\times (\cD\cup (\overline{\cD}\cap \partial \bR^d)))$.

We use the notation $a_+=\max\{a,0\}$ and $a_-=\max\{-a,0\}$ for $a\in \bR$ so that $a=a_+-a_-$.  Finally,  for a set $\Omega\subset \bR^{d+1}$ and any integrable function $f$ on $\Omega$ with respect to some Borel measure $\omega$, we write
$$
\fint_\Omega f\ \omega(dz) =\frac 1 {\omega(\Omega)}\int_\Omega f\, \omega(dz), \quad \text{where} \quad \omega(\Omega) = \int_{\Omega} \omega(dz).
$$
\subsection{Main theorems} As in \cite{Dong-Phan,D-P}, we impose the following partially  VMO condition on the leading coefficients.
\begin{assumption}[$\gamma_0, R_0$]
            \label{assump1}
For any $r\in (0, R_0]$ and $z_0=(z_0',x_d) \in \bR^d \times \overline{\bR}_+$,
we have
\[
\sup_{i,j}\fint_{Q^+_{r}(z_0)} |a_{ij}(t,x) - [a_{ij}]_{r,z_0}(x_d)|\, \mu(dz)\\
\le \gamma_0,
\]
where $\mu(dz) = x_d^\alpha\, dtdx$, $[a_{ij}]_{r,z_0}(x_d)$ is the average of $a_{ij}$ with respect to $(t, x')$  in $Q'_{r}(z_0')$:
$$
[a_{ij}]_{r,z_0}(x_d) = \fint_{Q'_{r}(z_0')}a_{ij}(t,x', x_d) \,dx'\, dt.
$$
\end{assumption} \noindent
In the special case that the coefficients $(a_{ij})$ only depend on the $x_d$ variable,  no regularity assumption is required on them as Assumption \ref{assump1} ($\gamma_0, R_0$) is always satisfied.

Our first main result is about the existence, uniqueness, and global regularity estimates of solutions to the divergence form equation \eqref{eq3.23}.
\begin{theorem} \label{thm2}
Let $\alpha \in (-1, \infty)$, $\kappa \in (0,1)$, $R_0 \in (0, \infty)$, and $p \in (1,\infty)$.  Then there exist $\gamma_0=\gamma_0(d,\kappa,\alpha,p) \in (0,1)$ and $\lambda_0=\lambda_0(d,\kappa,\alpha, p)\ge 0$ such that the following assertions hold. Suppose that \eqref{ellipticity}  and \textup{Assumption \ref{assump1} ($\gamma_0, R_0$)} are satisfied. If $u\in \cH^1_p(\Omega_T,\mu)$ is a weak solution of
\eqref{eq3.23} for some $\lambda\ge \lambda_0 R_0^{-2}$, $f\in L_p(\Omega_T,\mu)$, and $F \in L_p(\Omega_T, \mu)^d$,  then we have
\begin{equation} \label{main-thm-est}
\|D u\|_{L_p(\Omega_T,\mu)}+\sqrt{\lambda} \|u\|_{L_p(\Omega_T,\mu)}
\leq N \|F\|_{L_p(\Omega_T, \mu)} + N\|f\|_{L_p(\Omega_T,\mu)},
\end{equation}
where $N=N(d,\kappa,\alpha,p)>0$. Moreover, for any $\lambda > \lambda_0 R_0^{-2}$, $f\in L_p(\Omega_T,\mu)$, and $F \in L_p(\Omega_T,\mu)^d$, there exists a unique weak solution $u\in \cH^1_p(\Omega_T, \mu)$ to \eqref{eq3.23}.
\end{theorem}

In the next result, we give a local boundary estimate in a half cylinder. Consider
 \begin{equation} \label{eqn-entire}
\left\{
\begin{aligned}
x_d^\alpha u_t- D_i\big(x_d^\alpha (a_{ij} D_j u - F_i)\big)  & =  x_d^\alpha f\\
\lim_{x_d \rightarrow 0^+}x_d^\alpha (a_{dj} D_j u - F_d)  &= 0
\end{aligned}\right.  \quad \text{in} \  Q_2^+.
\end{equation}
Let $p  \in [1,\infty)$ and $p^*\in [1,p)$ satisfy
\begin{equation}
                                \label{eq3.19}
\begin{cases}
(d+2+\alpha_+)/p^*\le 1+(d+2+\alpha_+)/p\quad\text{when}\ p^*>1\\
(d+2+\alpha_+)/p^*< 1+(d+2+\alpha_+)/p\quad\text{when}\ p^*=1,
\end{cases}
\end{equation}
if $d \geq 2$ or $\alpha=0$, and
\begin{equation}
                                \label{eq3-2.19}
\begin{cases}
(4+\alpha_+)/p^*\le 1+(4+\alpha_+)/p\quad\text{when}\ p^*>1\\
(4+\alpha_+)/p^*< 1+(4+\alpha_+)/p\quad\text{when}\ p^*=1,
\end{cases}
\end{equation}
if $d =1$  and $\alpha\neq 0$. Note that the condition on $p^*$ is used in a weighted parabolic Sobolev embedding result.  See Lemma \ref{lem2.2} below.

\begin{corollary} \label{main-thrm}
Let $\alpha \in (-1, \infty)$, $\kappa \in (0,1)$, $R_0\in (0,\infty)$,  $1<p_0<p<\infty$, and $p^*\in [1,p)$ satisfy \eqref{eq3.19}-\eqref{eq3-2.19}. Then there exists $\gamma_0=\gamma_0(d,\kappa,\alpha,p_0,p) \in (0,1)$ such that the following assertion holds.  Suppose that \eqref{ellipticity} and  \textup{Assumption} \ref{assump1} ($\gamma_0, R_0$) are satisfied. If $u\in \cH^1_{p_0}(Q_2^+, \mu)$ is a weak solution of \eqref{eqn-entire}, $F\in L_p(Q_2^+,\mu)^{d}$, and $f\in L_{p^*}(Q_2^+, \mu)$, then $u \in \cH_p^1(Q_1^+, \mu)$ and
\begin{align} \label{main-thm-estb}
& \|u\|_{L_p(Q_{1}^+, \mu)}+ \|D u\|_{L_p(Q_{1}^+, \mu)} \nonumber\\
&\leq  N \|u\|_{L_1(Q_2^+, \mu)} + N \|Du\|_{L_1(Q_2^+, \mu)} + N\|F\|_{L_p(Q_2^+,\mu)} + N\|f\|_{L_{p^*}(Q_2^+, \mu)},
\end{align}
where $N=N(d,\kappa,\alpha,p_0,p, p^*, R_0)>0$.
\end{corollary}
We conjecture that for any $d\ge 1$ and $\alpha\in (-1,\infty)$, the above corollary still holds when $p^*$ satisfies \eqref{eq3.19}.

In this paper, we also show that Theorem \ref{thm2} can be extended
to the setting of weighted mixed-norm spaces.  The result is of interest because the inhomogeneous terms $F$ and $f$ could behave anisotropically. For $p \in (1, \infty)$, a locally integrable function $\omega : \bR^d_+ \rightarrow \bR_+$ is said to be in $A_p(\bR^d_+, \mu)$ Muckenhoupt class of weights if \begin{equation*}
[\omega]_{A_p(\bR^d_+, \mu)}: =
\sup_{r >0,x \in \overline{\bR^d_+}} \Big(\fint_{B^+_r(x)} \omega(y)\, \mu(dy) \Big)\Big(\fint_{B^+_r(x)} \omega(y)^{\frac{1}{1-p}}\, \mu(dy) \Big)^{p-1}<\infty.
\end{equation*}
Similarly, a locally integrable function $\omega : \bR \rightarrow \bR_+$ is said to be in $A_p(\bR)$ Muckenhoupt class of weights if
\begin{equation*}
[\omega]_{A_p(\bR)}: =
\sup_{r >0,t \in \bR}
\Big(\fint_{t-r^2}^{t+r^2} \omega(s)\, ds \Big)\Big(\fint_{t-r^2}^{t+r^2} \omega(s)^{\frac{1}{1-p}}\, ds \Big)^{p-1}<\infty.
\end{equation*}
\begin{theorem} \label{thm1.4}
Let $\alpha \in (-1, \infty)$, $\kappa \in (0,1)$,  $R_0 \in (0, \infty)$, $p, q, K \in (1,\infty)$, $\omega_0\in A_q(\bR)$, $\omega_1\in A_p(\bR^d_+, \mu)$, and $\omega=\omega_0(t)\omega_1(x)$, such that
$$
 [\omega_0]_{A_q(\bR)}\le K,\quad [\omega_1]_{A_p(\bR^d_+, \mu)}\le K.
$$
Then there exist
$$
\gamma_0=\gamma_0(d, \kappa, \alpha, p, q,  K)\in (0,1) \quad \text{and}\quad
\lambda_0=\gamma_0(d, \kappa, \alpha, p, q, K)\ge 0,
$$
such that the following assertions hold. Suppose that \eqref{ellipticity} and \textup{Assumption \ref{assump1} ($\gamma_0, R_0$)} are satisfied. If $u\in \cH^1_{q,p}(\Omega_T,\omega d\mu)$ is a weak solution of
\eqref{eq3.23} for some $\lambda\ge \lambda_0 R_0^{-2}$, $f\in L_{q,p}(\Omega_T,\omega\,d\mu)$, and $F \in L_{q,p}(\Omega_T,\omega\, d\mu)^d$, then we have
\begin{align} \label{main-thm-estc}
&\|D u\|_{L_{q,p}(\Omega_T,\omega\,d\mu)}+\sqrt{\lambda} \|u\|_{L_{q,p}(\Omega_T,\omega\,d\mu)}\nonumber\\
&\leq N \|F\|_{L_{q,p}(\Omega_T,\omega\,d\mu)} + N\|f\|_{L_{q,p}(\Omega_T,\omega\,d\mu)},
\end{align}
where $N=N(d, \kappa, \alpha, p, q,  K)>0$.
Moreover, for any $\lambda > \lambda_0 R_0^{-2}$, $f\in L_{q, p}(\Omega_T, \omega d\mu)$, and $F \in L_{q,p}(\Omega_T, \omega d\mu)^d$,  there exists a unique weak solution  $u\in \cH^1_{q, p}(\Omega_T, \omega d\mu)$ to \eqref{eq3.23}.
\end{theorem}

Next, we state the main results for  non-divergence form equations. Besides the regularity assumption on $(a_{ij})$ as in Assumption \ref{assump1}, we impose  similar conditions on the coefficients $a_0$ and $c_0$.
\begin{assumption}[$\gamma_0, R_0$]
            \label{assump2}
For any $r\in (0, R_0]$ and $z_0=(z_0',x_d) \in \bR\times \overline{\bR^d_+}$, 
we have
\[
\begin{split}
&  \sup_{i,j}\fint_{Q^+_{r}(z_0)} |a_{ij}(t,x) - [a_{ij}]_{r,z_0}(x_d)|\, \mu(dz)\\
& + \fint_{Q^+_{r}(z_0)}\Big( |a_{0}(t,x) - [a_0]_{r, z_0}(x_d)| +  |c_{0}(t,x) - [c_0]_{r,z_0}(x_d)|  \Big) \, \mu(dz) \le \gamma_0,
\end{split}
\]
where $[a_{ij}]_{r,z_0}(x_d)$, $[a_0]_{r, z_0}(x_d)$, and $[c_0]_{r,z_0}(x_d)$ are respectively the average of $a_{ij}$, $a_0$, and $c_0$ with respect to $(t, x')$ in $Q'_{r}(z_0')$ as defined in Assumption \ref{assump1}.
\end{assumption}
We also need the following definition which is used in a weighted Hardy inequality (cf. \cite[Lemma 2.2]{D-P}).
\begin{definition} Let $\alpha \in (-1,\infty)$ and $p \in (1,\infty)$, we say that the weight $\omega: \bR_+ \rightarrow \bR_+$ is in $M_{p}(\mu)$ if
\[
[\omega]_{M_{p}(\mu)} = \sup_{r>0} \left(\int_r^\infty y^{- p(\alpha +1)} \omega(y) \,\mu(dy) \right)^{\frac 1 p} \left( \int_0^r \omega(y)^{-\frac{1}{p-1}} \,\mu(dy) \right)^{1-\frac{1}{p}} < \infty,
\]
where $\mu(dy) = y^\alpha\ dy$ for $y \in \bR_+$.
\end{definition}

Define $W^{1,2}_{q,p}(\Omega_T, \omega\, d\mu)$ to be the weighted mixed-norm Sobolev space equipped with the norm
\[
\begin{split}
\norm{u}_{W^{1,2}_{q, p} (\Omega_T, \omega\, d\mu)} & = \norm{u}_{L_{q, p}(\Omega_T, \omega\,d\mu)} + \norm{u_t}_{L_{q, p} (\Omega_T, \omega\, d\mu)} \\
& \qquad + \norm{Du}_{L_{q, p} (\Omega_T, \omega\, d\mu)} + \norm{D^2u}_{L_{q, p}(\Omega_T, \omega\, d\mu)}.
\end{split}
\]
When $p =q$ and $\omega \equiv 1$, we write $W^{1,2}_p(\Omega_T, \mu) = W^{1,2}_{p,p}(\Omega_T, d\mu)$. A function $u \in W^{1,2}_{q,p}(\Omega_T,  {\omega}\ d\mu)$ is said to be a strong solution to \eqref{non-div.eqn} if it satisfies the equation almost everywhere.  Our main result for the non-divergence form equation \eqref{non-div.eqn}--\eqref{non-div.bry} is the following theorem.
\begin{theorem}  \label{para.theorem} Let $\alpha \in (-1,\infty)$, $\kappa \in (0,1)$, $R_0 \in (0, \infty)$, $p, q, K \in (1, \infty)$.   Let $\omega_0\in A_q(\bR)$, $\omega_1 \in A_p(\bR^{d-1})$, $\omega_2 \in A_p(\bR_+, \mu) \cap M_{p}(\mu)$, and $\omega(t,x) = \omega_0(t) \omega_1(x') \omega_2(x_d)$, such that
$$
 [\omega_0]_{A_q(\bR)}\le K,\quad
[\omega_1]_{A_p(\bR^{d-1})}\le K,\quad
[\omega_2]_{A_p(\bR_+, \mu)}\le K,\quad
[\omega_2]_{M_{p}(\mu)}\le K.
$$
Then there exist
\[
\gamma_0  = \gamma_0(d, \kappa, \alpha, p, q,  K) \in (0,1)\quad\text{and}\quad
\lambda_0  = \lambda_0(d,\kappa,\alpha, p,q, K) \geq 0
\] 
such that the following assertions hold. Suppose that \eqref{ellipticity}, \eqref{a-c-eq},  \eqref{extension-type-matrix} and \textup{Assumption \ref{assump2} ($\gamma_0, R_0)$} are satisfied.  If  $u \in W^{1,2}_{q,p}(\Omega_T,  {\omega}\ d\mu)$ is a strong solution of \eqref{non-div.eqn}-\eqref{non-div.bry} with $f \in L_{q,p}(\Omega_T, \omega\, d\mu)$ and $\lambda \ge \lambda_0 R_0^{-2}$, then
\begin{equation*} 
\begin{split}
& \norm{u_t}_{L_{q,p}(\Omega_T, \omega\, d\mu)} + \norm{D^2u}_{L_{q,p}(\Omega_T, \omega\, d\mu)} + \norm{D_du/x_d}_{L_{q,p}(\Omega_T, \omega\, d\mu)} \\
& \quad + \sqrt{\lambda} \norm{Du}_{L_{q,p}(\Omega_T, \omega\, d\mu)} + \lambda \norm{u}_{L_{q, p}(\Omega_T, \omega\, d\mu)} \leq  N \norm{f}_{L_{q,p}(\Omega_T, \omega\, d\mu)},
\end{split}
\end{equation*}
where $N=N(d,\kappa,\alpha,p,q,K)>0$. Moreover,  for any $f \in L_{q,p}(\Omega_T, \omega\, d\mu)$ and $\lambda > \lambda_0 R_0^{-2}$,  there is a unique strong solution $u \in W^{1,2}_{q,p}(\Omega_T,  {\omega}\ d\mu)$ of \eqref{non-div.eqn}-\eqref{non-div.bry}.
\end{theorem}
\begin{remark} 
As a typical example, in Theorem \ref{para.theorem} we can take  the power weight $\omega_2(x_d) = x_d^\beta$. It is easily seen that for any $\beta \in (-\alpha -1, (\alpha +1) (p-1))$, we have $\omega_2 \in A_p(\bR_+, \mu) \cap M_{p}(\mu)$.  In the special case when $\alpha=0$ and $\beta\in (-1,p-1)$, a similar result was proved in \cite{DKZ16} when the coefficients are measurable in the time variable and have small mean oscillations in the spatial variables, by using a different argument.
\end{remark}

Once Lipschitz  and Schauder estimates in Section \ref{div-sim-sec} and Theorem \ref{thm2} are proved, Theorem \ref{para.theorem} can be proved  by using the same argument as in  \cite{D-P}. To keep the paper  within a reasonable length, we skip the proof of Theorem \ref{para.theorem} and refer the reader to \cite{D-P} for details.

Similarly to Corollary \ref{main-thrm}, we also obtain the following  local boundary estimate for solutions of \eqref{non-div.eqn} in $Q_2^+$.

\begin{corollary} \label{cor1.8}
Let $\alpha \in (-1, \infty)$, $\kappa \in (0,1)$, $R_0\in (0,\infty)$, and $1<p_0<p<\infty$. Then there exists $\gamma_0=\gamma_0(d,\kappa,\alpha,p_0,p) \in (0,1)$ such that the following assertion holds.  Suppose that \eqref{ellipticity}, \eqref{a-c-eq},  \eqref{extension-type-matrix}, and  \textup{Assumption} \ref{assump2} ($\gamma_0, R_0$) are satisfied. If $u\in W^{1,2}_{p_0}(Q_2^+, \mu)$ is a strong solution of
\begin{equation*} 
\left\{
\begin{aligned}
a_0u_t- a_{ij}D_{ij}u - \frac \alpha {x_d} a_{dd}D_d u+ c_0 u   &=  f\\
\lim_{x_d \rightarrow 0^+}x_d^\alpha a_{dd} D_d u  &= 0
\end{aligned} \right.
\quad \text{in}\ \ Q_2^+
\end{equation*}
and  $f\in L_p(Q_2^+,\mu)$, then  we have $u \in W_p^{1,2}(Q_1^+, \mu)$ and
\begin{align} \label{eq6.27}
\|u\|_{W^{1,2}_p(Q_{1}^+, \mu)}\leq  N \|u\|_{W^{1,2}_1(Q_2^+, \mu)} + N\|f\|_{L_p(Q_2^+,\mu)},
\end{align}
where $N=N(d,\kappa,\alpha,p_0,p, R_0)>0$.
\end{corollary}

Using the above results for parabolic equations, we can directly derive similar results for elliptic equations by viewing solutions to elliptic equations as steady state solutions of the corresponding parabolic equations. See, for example, the proofs of \cite[Theorem 2.6]{Krylov} and \cite[Theorem 1.2]{D-P}. We only present here a result  of the local boundary estimate for weak solutions. Consider
\begin{equation} \label{eq.0612}
\left\{
\begin{aligned}
 - D_i(x_d^\alpha[ a_{ij}(x) D_j u - F_i]) & = x_d^\alpha f\\
 \lim_{x_d \rightarrow 0^+}x_d^\alpha (a_{dj}(x) D_j u - F_d)  &= 0
\end{aligned} \right. \quad \text{in} \  B_2^+,
\end{equation}
where $a_{ij}: B_2^+ \rightarrow \bR$, $F = (F_1, F_2,\ldots, F_d) : B_2^+  \rightarrow \bR^d$ and $f: B_2^+ \rightarrow \bR$ are given measurable functions. In this time-independent case, \eqref{ellipticity} and Assumption \ref{assump1} can be stated similarly. For each $p  \in (1,\infty)$, suppose that $\hat{p} \in [1,p)$ satisfies
\begin{equation} \label{eq1.0612}
                               \begin{cases}
(d+\alpha_+)/\hat{p} \le 1+(d + \alpha_+)/p\quad\text{when}\ \hat{p}>1,\\
(d+ \alpha_+)/\hat{p} < 1+(d+ \alpha_+)/p\quad\text{when}\ \hat{p}=1.
\end{cases}
\end{equation}
For $\Omega \subset \bR^d$, $W^{1}_{p}(\Omega, \mu)$ denotes the weighted Sobolev space consisting of all measurable functions $u: \Omega \rightarrow \bR$ such that $u, Du \in L_p(\Omega, \mu)$.

\begin{corollary}
                \label{ell-local-th}
Let $\alpha \in (-1, \infty)$, $\kappa \in (0,1)$, $R_0\in (0,\infty)$, $1<p_0<p<\infty$, and $\hat{p} \in [1,p)$ satisfy \eqref{eq1.0612}. Then there exists $\gamma_0=\gamma_0(d,\kappa,\alpha, p_0, p) \in (0,1)$ such that the following assertion holds.  Suppose that \eqref{ellipticity} and  \textup{Assumption} \ref{assump1} ($\gamma_0, R_0$) are satisfied. If $u\in W^{1}_{p_0}(B_2^+, \mu)$ is a weak solution of \eqref{eq.0612}, $F\in L_p(B_2^+,\mu)^{d}$, and $f\in L_{\hat{p}}(B_2^+, \mu)$, then $u \in W^1_{p}(B_1^+, \mu)$ and
\begin{align*}
  \|u\|_{W^1_p(B_{1}^+, \mu)} \leq  N \|u\|_{W^1_1(B_2^+, \mu)} + N\|F\|_{L_p(B_2^+,\mu)} + N\|f\|_{L_{\hat{p}}(B_2^+, \mu)},
\end{align*}
where $N=N(d,\kappa,\alpha,p_0,p, \hat{p}, R_0)>0$.
\end{corollary}

 The proof of  Corollary \ref{ell-local-th} is similar to that of Corollary \ref{main-thrm} by using the corresponding weighted embedding inequality. See Remark \ref{imbd-remark} (ii). Therefore, we also omit it.

\section{Weighted Sobolev inequalities and \texorpdfstring{$L_2$}{L2}-solutions} \label{inq-sec}
Our first result in this section is a weighted parabolic embedding lemma which will be used in the proof of Corollary \ref{main-thrm}. The range of $q^*$ below is optimal when $d\ge 2$. However, when $d=1$, we impose a slightly stronger condition.  In view of the classical parabolic Sobolev embedding when $\alpha=0$, we conjecture that this condition can be relaxed.

\begin{lemma}[Weighted parabolic imbedding]
                            \label{lem2.2}
Let $\alpha\in (-1,\infty)$ and $q,q^*\in (1,\infty)$ satisfy
\begin{equation}
                        \label{eq2.06}
\left\{
\begin{aligned}
(d+2+\alpha_+)/q\le 1+(d+2+\alpha_+)/q^* & \quad \text{if} \quad d \geq 2\\
(4+\alpha_+)/q\le 1+(4+\alpha_+)/q^* & \quad \text{if} \quad d =1.
\end{aligned} \right.
\end{equation}
Then for any $v\in \cH^1_q(Q_2^+,\mu)$, we have
\begin{equation}
                    \label{eq4.15}
\|v\|_{L_{q^*}(Q_2^+,\mu)}\le N\|v\|_{\cH^1_{q}(Q_2^+,\mu)},
\end{equation}
where $N=N(d,\alpha,q,q^*)>0$ is a constant. The result still holds when $q^*=\infty$ and the inequality in \eqref{eq2.06} is strict.
\end{lemma}

\begin{proof}
Note that the case when $d=1$ follows by considering $v (t,x_1)=v(t,x_1,x_2)$ with a dummy variable $x_2$ and using  the result when $d=2$. Hence, we only need to prove \eqref{eq4.15}  when $d \geq 2$.
It suffices to consider the case when $q^*>q$. Without loss of generality, we may assume that
\begin{equation}
                                \label{eq2.24}
v_t= D_i G_i+G_0/x_d+g
\end{equation}
in $Q_2^+$ in the weak sense and
$$
\|v\|_{L_q(Q_2^+,\mu)}+\|Dv\|_{L_q(Q_2^+,\mu)}
+\|G\|_{L_q(Q_2^+,\mu)}+\|g\|_{L_q(Q_2^+,\mu)}\le 1,
$$
where $G = (G_0, G_1,\ldots, G_d)$. Let $\widetilde Q=Q_{1/2}(0,0,\ldots,0, 3/2)$ and $\psi\in C_0^\infty(\widetilde Q)$ with unit integral.
For any $(t,x)\in Q_2^+$,  by the fundamental theorem of calculus,
\begin{align}
                        \label{eq10.17}
&v(t,x)-c\notag\\
&=\int_{\widetilde Q}\int_0^1 \Big(v_t(t(1-\theta^2)+s\theta^2,x(1-\theta)+y\theta)2\theta(s-t)\notag\\
&\quad +(D v)(t(1-\theta^2)+s\theta^2,x(1-\theta)+y\theta)\cdot (y-x)\Big)\psi(s,y)\,d\theta \,ds\,dy\notag\\
&:=I_1+I_2,
\end{align}
where
$$
c=\int_{\widetilde Q} v(s,y)\psi(s,y)\,ds\,dy.
$$
Let $\hat{x}=x(1-\theta)+y\theta$ and $\tau=t(1-\theta^2)+s\theta^2$. Clearly,
\begin{equation}
                                \label{eq11.03}
(|\hat{x}-x|^2+|\tau-t|)^{1/2}=(|x-y|^2+|t-s|)^{1/2}\theta\le N\theta.
\end{equation}
It then follows from \eqref{eq2.24} that
\begin{align}
                    \label{eq10.18}
I_1=2\int_{\widetilde Q}\int_0^1 \Big(g(\tau,\hat{x})+(D_i{G_i})(\tau, \hat{x})+\frac{G_0(\tau, \hat{x})}{\hat{x}_d}\Big)
\theta(s-t)\psi(s,y)\,d\theta \,ds\,dy.
\end{align}
Since $y \in \widetilde{Q}$, we have $y_d \geq 1$ and thus
\begin{equation}
                    \label{eq3.26}
|x-\hat{x}|=|x-y|\theta \leq N\theta y_d \le  N \hat{x}_d
\quad \text{and}\quad x_d\le N\hat x_d.
\end{equation}
Moreover,
$$
(D_i{G_i})(\tau, \hat{x})
=D_{y_i}{G_i}(\tau, \hat{x})\theta^{-1}.
$$
Therefore, from \eqref{eq10.18} and integration by parts, we deduce
\begin{align}
                                \label{eq10.39}
|I_1|\le N\int_{\widetilde Q}\int_0^1 \big(|g(\tau, \hat{x})|\theta+|{G}(\tau, \hat{x})|+|{G_0}(\tau, \hat{x})||x-\hat{x}|^{-1}\theta\big)
|s-t|\,d\theta \,ds\,dy.
\end{align}
Combining \eqref{eq10.17} and \eqref{eq10.39}, we obtain
\begin{align*}
&|v(t,x)-c|\\
&\le N\int_{\widetilde Q}\int_0^1 \big(|g(\tau, \hat{x})|\theta+|{G}(\tau,\hat{x})|+|{G_0}(\tau, \hat{x})||x-\hat{x}|^{-1}\theta+|Dv(\tau,\hat{x})|\big)
\,d\theta \,ds\,dy\notag\\
&\le N\int_{Q_2^+}\int_0^1 \theta^{-d-2}\big(|g(\tau,\hat{x})|\theta+|{G}(\tau,\hat{x})|\notag\\
&\quad+|G_0(\tau,\hat{x})||x-\hat{x}|^{-1}\theta+|Dv(\tau,\hat{x})|\big)\,
\chi_{\{(|x-\hat{x}|^2+|t-\tau|)^{1/2}\le N\theta ,x_d\le N\hat x_d\}}d\theta \,d\tau\,d\hat{x}\notag\\
&\le N\int_{Q_2^+}\big(|g(\tau,\hat{x})|(|x-\hat{x}|^2+|t-\tau|)^{-d/2}\notag\\
&\qquad +|G_0(\tau,\hat{x})||x-\hat{x}|^{-1}(|x-\hat{x}|^2+|t-\tau|)^{-d/2}
\notag\\
&\quad+(|{G}(\tau,\hat{x})|
+|Dv(\tau,\hat{x})|)(|x-\hat{x}|^2+|t-\tau|)^{-(d+1)/2}\big)\chi_{\{x_d\le N\hat x_d\}} \,d\tau\,d\hat{x},
\end{align*}
where we used $dy=\theta^{-d}\,d\hat{x}$, $d\tau=\theta^{-2}\,ds$, \eqref{eq11.03}, and \eqref{eq3.26} in the third inequality.

We apply Young's inequality for convolutions with respect to the time variable to get that for any $x\in B_2^+$,
\begin{align}
                \label{eq3.20}
&\|v(\cdot,x)-c\|_{L_{q^*}((-4,0))}\le N\int_{B_2^+}\big(\|g(\cdot,\hat{x})\|_{L_{q}((-4,0))}|x-\hat{x}|
+\|{G}(\cdot,\hat{x})\|_{L_{q}((-4,0))}
\notag\notag\\
&\quad \qquad+
\|Dv(\cdot,\hat{x})\|_{L_{q}((-4,0))}\big)|x-\hat{x}|^{-d-1+2/\ell}\chi_{\{x_d\le N\hat x_d\}}\,d\hat{x},
\end{align}
where $\ell\in (1,\infty)$ satisfies
\begin{equation}
                                \label{eq1.39}
1/\ell+1/q=1+1/q^*
\end{equation}
and $\ell d>2$ which always holds because $\ell>1$ and $d\ge 2$. Similarly, using Young's inequality in $x'$ to get that for any $x_d\in (0,2)$,
\begin{align}   \label{eq1.47}
&\|v(\cdot,\cdot,x_d)-c\|_{L_{q^*}(Q_2')}\notag \\
& \le N\int_{0}^{2}\big(\|g(\cdot,\cdot, \hat{x}_d)\|_{L_{q}(Q_2')}|x_d-\hat{x}_d|
+\|{G}(\cdot,\cdot, \hat{x}_d)\|_{L_{q}(Q_2')}
\notag\\
&\qquad
+\|Dv(\cdot,\cdot, \hat{x}_d)\|_{L_{q}(Q_2')}\big)|x_d-\hat{x}_d|^{-(d+1)(1-1/\ell)}\chi_{\{x_d\le N\hat x_d\}}\,d\hat{x}_d,
\end{align}
where we used $(d+1-2/\ell)\ell>d-1$ which holds true as $\ell >1$. In the sequel, we discuss two cases: $\alpha\ge 0$ and $\alpha\in (-1,0)$.\\
\noindent
{\bf Case I: $\alpha\ge 0$.} We first consider the case when $q^* <\infty$. Multiplying both sides of \eqref{eq1.47} by $x_d^{\alpha/q*}$, we get
\begin{align*}
&x_d^{\alpha/q*}\|v(\cdot,\cdot,x_d)-c\|_{L_{q^*}(Q_2')}\notag\\
& \le Nx_d^{\alpha/q*}\int_{0}^{2}\big(\|g(\cdot,\cdot, \hat{x}_d)\|_{L_{q}(Q_2')}|x_d-\hat{x}_d|
+\|{G}(\cdot,\cdot, \hat{x}_d)\|_{L_{q}(Q_2')}\notag\\
&\qquad
+\|Dv(\cdot,\cdot, \hat{x}_d)\|_{L_{q}(Q_2')}\big)|x_d-\hat{x}_d|^{-(d+1)(1-1/\ell)}
\chi_{\{x_d\le N\hat{x}_d\}}\,d\hat{x}_d\notag\\
& \le N\int_{0}^{2}\big(\|g(\cdot,\cdot,\hat{x}_d)\|_{L_{q}(Q_2')}|x_d-\hat{x}_d|
+\|{G}(\cdot,\cdot,\hat{x}_d)\|_{L_{q}(Q_2')}\notag\\
&\qquad
+\|Dv(\cdot,\cdot, \hat{x}_d)\|_{L_{q}(Q_2')}\big)\hat{x}_d^{\alpha/q}
|x_d-\hat{x}_d|^{-(d+1)(1-1/\ell)+\alpha/q^*-\alpha/q}
\,d\hat{x}_d.
\end{align*}
Since both $x_d$ and $\hat x_d$ are bounded,
we can apply the Hardy--Littlewood-Sobolev inequality for fractional integration in $x_d$ to obtain
\begin{align}
                                \label{eq1.33}
&\|v-c\|_{L_{q^*}(Q^+_2,\mu)}
\le N\|g\|_{L_{q}(Q^+_2,\mu)}
+N\|{G}\|_{L_{q}(Q^+_2,\mu)}
+N\|Dv\|_{L_{q}(Q^+_2,\mu)}
\end{align}
provided that
$$
(d+1)(1-1/\ell)-\alpha/q^*+\alpha/q\le 1+1/q^*-1/q.
$$
From \eqref{eq1.39}, we see that this condition is equivalent to \eqref{eq2.06}.

When $q^*=\infty$, we have $\ell=p=q/(q-1)$. Thus, if the inequality \eqref{eq2.06} is strict, we also get \eqref{eq1.33} by using H\"older's inequality.
From \eqref{eq1.33} and the definition of $c$, we easily get \eqref{eq4.15}.

\noindent
{\bf Case II: $\alpha\in (-1,0)$.}  For the case $q^* <\infty$, we will apply the generalized Hardy--Littlewood Sobolev inequality (see \cite[Theorem 6]{HL28} or \cite[Theorem B]{SW58}) to conclude \eqref{eq1.33}, which gives \eqref{eq4.15}. Indeed, in terms of the notation in \cite[Theorem 6]{HL28}, we choose
\[
 r=q,\quad s=\frac{q^*}{q^*-1},\quad
 h=\frac{ \alpha}{ q^*},\quad k=-\frac \alpha {q^*},\quad \lambda=2-\frac 1 s-\frac 1 r.
\]
Then it is easily seen that the conditions in there are satisfied.  Let
\[
\begin{split}
f(\hat{x}_d) = \big(\|g(\cdot,\cdot, \hat{x}_d)\|_{L_{q}(Q_2')}
+\|{G}(\cdot,\cdot, \hat{x}_d)\|_{L_{q}(Q_2')}
+\|Dv(\cdot,\cdot, \hat{x}_d)\|_{L_{q}(Q_2')}\big) \chi_{(0,2)}(\hat{x}_d).
\end{split}
\]
As both $x_d$ and $\hat x_d$ are bounded, if
\begin{equation} \label{0526-eqn}
(d+1)(1-1/\ell) \le \lambda-h-k,
\end{equation}
then by \eqref{eq1.47} we see that for any $g \in L_{s}((0,2))$
\[
\begin{split}
&\left| \int_0^{2} x_d^{\alpha/q*}\|v(\cdot,\cdot,x_d)-c\|_{L_{q^*}(Q_2')} g(x_d) \, dx_d \right| \\
&  \leq N \int_0^{2} \int_0^2 \frac{f(\hat{x}_d) \hat{x}_d^{h} |g(x_d)|}{\hat{x}_d^{h}|x_d - \hat{x}_d|^{\lambda -h -k} x_d^{k}}\, d\hat{x}_d \, dx_d.
\end{split}
\]
From this, we apply  \cite[Theorem 6]{HL28} to get
\[
\begin{split}
& \left| \int_0^{2}  x_d^{\alpha/q*}\|v(\cdot,\cdot,x_d)-c\|_{L_{q^*}(Q_2')} g(x_d) \, dx_d \right|\\
& \leq N \left(\int_0^2 f^q(\hat{x}_d) \hat{x}_d^{hq}\, d\hat{x}_d \right)^{1/q} \|g\|_{L_s((0,2))}\\
& \leq N  \left(\int_0^2 f^q(\hat{x}_d) \hat{x}_d^{\alpha}\, d\hat{x}_d \right)^{1/q}\|g\|_{L_s((0,2))},
\end{split}
\]
where we used the fact that $\hat{x}_d^{hq} \leq N \hat{x}_d^\alpha$ for any $\hat{x}_d \in (0,2)$ because $\alpha <0$.
Then by the duality, we obtain \eqref{eq1.33} when $q^* <\infty$. Because of \eqref{eq1.39}, the condition \eqref{0526-eqn} is equivalent to
\begin{equation}
                    \label{eq2.17}
(d+2)/q\le 1+(d+2)/q^*,
\end{equation}
which is \eqref{eq2.06} when $\alpha <0$.
When $q^*=\infty$ and the inequality \eqref{eq2.17} is strict, we also have \eqref{eq1.33} by using H\"older's inequality.
The lemma is proved.
\end{proof}
\begin{remark} \label{imbd-remark} \textup{(i)}
In view of the additional factors in the $g$ terms in \eqref{eq3.20} and \eqref{eq1.47}, it is possible to relax the integrability condition on $g$ in Lemma \ref{lem2.2}: we only need $g\in L_{\widetilde q}(Q_2^+,\mu)$, where $\widetilde q\in (1,q)$ satisfies
$$
(d+2+\alpha_+)/\widetilde q\le 2+(d+2+\alpha_+)/q^*
$$
when $d\ge 2$. However, this will not be used in the proofs of our main results.\\
\noindent
\textup{(ii)} In  the time-independent case,  \eqref{eq3.20} is not needed. Therefore, with  a minor modification of the proof, we also have the embedding:
\[
\|u\|_{L_{q^*}(B_2^+, \mu)} \leq N \|u\|_{W^1_q(B_2^+, \mu)} \quad \text{for all} \ u \in W^{1}_q(B_2^+, \mu)
\]
with $q, q^* \in (1, \infty)$ satisfying
\[
(d+\alpha_+)/q\le 1+(d+\alpha_+)/q^*.
\]
The result still holds when $q > d+\alpha_+$ and $q^*=\infty$. See \cite[Theorem 6]{Ha} for a different proof in a more general setting.
\end{remark}

We also need a weighted parabolic embedding result for functions in the energy space, which will be used in the proof of Lemma \ref{lem2} when we apply the Moser iteration.
\begin{lemma}
                                    \label{Sobolev-imbed}
Let $\alpha \in (-1, \infty)$, $ l_0 = \frac{d+\alpha_{+}+2}{d+\alpha_{+}}$ if $d + \alpha_+ >2$ and $l_0 \in (1, 2)$ be any number if $d + \alpha_+ \leq 2$.  Then there exists a constant $N = N(d, l_0, \alpha)$ such that
\[
\begin{split}
& \bigg(\fint_{Q_r^+(z_0)} |u(t,x)|^{2l_0} \,{\mu(dz)} \bigg)^{1/l_0} \\
& \leq N \sup_{t\in (t_0 -r^2, t_0)} \fint_{B_r^+(x_0)} |u(t,x)|^2 \,\mu(dx) + N r^2\fint_{Q_r^+(z_0)} |D u(t,x)|^2\,{\mu(dz)},
\end{split}
\]
for every $z_0 = (t_0, x_0) \in \overline{{\bR}^{d+1}_{+}}$, $r >0$, and
$$
u \in L_\infty((t_0-r^2, t_0); L_2(B_r^+(x_0), \mu)) \cap L_2((t_0-r^2, t_0);  W^{1}_2(B_r^+(x_0), \mu)).
$$
\end{lemma}
\begin{proof}  Let $\Gamma = (t_0-r^2, t_0)$, and let $ \kappa_0  = \frac{2}{2-l_0} \in (1, \infty)$. By Remark \ref{imbd-remark} (ii) (see also \cite[Theorem 2.4]{Sire-1}) and after rescaling,  we have the following weighted Sobolev inequality:
\begin{align} \notag
&\bigg( \fint_{B_r^+(x_0)} |u(t,x)|^{\kappa_0} \,\mu(dx) \bigg)^{\frac{1}{\kappa_0}}\\ \label{em-06-27}
&\leq N r \bigg( \fint_{B_r^+(x_0)} |D u(t,x)|^2 \,\mu(dx) \bigg)^{\frac{1}{2}} +N \bigg( \fint_{B_r^+(x_0)} |u(t,x)|^2 \,\mu(dx) \bigg)^{\frac{1}{2}},
\end{align}
where $N=N(d, l_0,\alpha)>0$. This together with H\"{o}lder's inequality gives
\[
\begin{split}
&  \fint_{B_r^+(x_0)} |u(t,x)|^{2l_0} \,\mu(dx) \\
 &  \leq \bigg( \fint_{B_r^+(x_0)} |u(t,x)|^{2} \,\mu(dx) \bigg)^{1-\frac{2}{\kappa_0}}  \bigg( \fint_{B_r^+(x_0)} |u(t,x)|^{\kappa_0} \,\mu(dx) \bigg)^{\frac{2}{\kappa_0}} \\
 & \leq N  \bigg(\sup_{t\in \Gamma} \fint_{B_r^+(x_0)} |u(t, x)|^2 \,\mu(dx)\bigg)^{1-\frac{2}{\kappa_0}}\\
  &\qquad \cdot \bigg( r^{2}\fint_{B_r^+(x_0)} |D u(t,x)|^{2} \,\mu(dx) + \fint_{B_r^+(x_0)} |u(t,x)|^{2} \,\mu(dx) \bigg).
\end{split}
\]
Now by integrating with respect $t$ on $\Gamma$ and using Young's inequality, we obtain
\begin{equation*} 
\begin{split}
& \fint_{Q_r^+(z_0)} |u(t,x)|^{2l_0} \,{\mu} (dz) \\
& \leq
N\Big( \sup_{t\in \Gamma} \fint_{B_r^+(x_0)} |u(t, x)|^2 \,\mu(dx) \Big)^{l_0}+ Nr\Big(  \fint_{Q_r^+(z_0)} |D u(t,x)|^{2} \,{\mu}(dz) \Big)^{l_0}.
\end{split}
\end{equation*}
The lemma is then proved.
\end{proof}

Finally, we conclude this section with the following useful result on the existence and uniqueness of $L_2$-solutions of a class of equations that are slightly more general than \eqref{eq3.23}. The result is considered as a special case of Theorem \ref{thm2} when $p =2$, but no regularity requirements are imposed on the coefficients.
\begin{lemma} \label{L-2-lemma} Let $\alpha \in (-1, \infty)$,  $\lambda > 0$, and let $(a_{ij})$, $a_0$, and $c_0$ be measurable functions defined on $\Omega_T$ such that \eqref{ellipticity} and \eqref{a-c-eq} are satisfied. Then for each $F \in L_2(\Omega_T,\mu)^d$ and $f \in  L_2(\Omega_T,\mu)$, there exists a unique weak solution $u\in \cH^1_2(\Omega_T,\mu)$ to
\begin{equation} \label{eq3-b.23}
\left\{
\begin{aligned}
 x_d^\alpha \big(a_0(t,x) u_t + \lambda c_0(t,x) u \big) - D_i\big(x_d^\alpha [a_{ij}(t,x) D_j u - F_i]\big)  & = \sqrt{\lambda} x_d^\alpha f\\
\lim_{x_d \rightarrow 0^+}x_d^\alpha \big(a_{dj}(t,x) D_j u - F_d\big)  &= 0
\end{aligned}  \right.
\end{equation}
in  $\Omega_T$. Moreover,
\begin{equation} \label{L2-lemma-est}
\|D u\|_{L_2(\Omega_T,\mu)}+\sqrt{\lambda} \|u\|_{L_2(\Omega_T,\mu)}
\leq N \|F\|_{L_2(\Omega_T, \mu)} +N\|f\|_{L_2(\Omega_T,\mu)},
\end{equation}
where $N = N(\kappa)$.
\end{lemma}
\begin{proof}
We first prove the a priori estimate \eqref{L2-lemma-est}. Let $u\in \cH^1_2(\Omega_T,\mu)$ be a weak solution of \eqref{eq3-b.23}. By multiplying the equation \eqref{eq3-b.23} with $u$ and using integration by parts and \eqref{ellipticity}, we obtain
\[
\begin{split}
& \sup_{t\in (-\infty, T)}\int_{\bR^d_+} |u(t,x)|^2 \,\mu(dx) + \int_{\Omega_T} |Du|^2 \,\mu(dz) + \lambda \int_{\Omega_T} |u(z)|^2 \,\mu(dz) \\
& \leq N\int_{\Omega_T} |F(z)| |Du(z)| \,\mu(dz) + N\lambda^{1/2} \int_{\Omega_T} |f(z)| |u(z)| \,\mu(dz).
\end{split}
\]
Then by Young's inequality, we obtain \eqref{L2-lemma-est}.

From \eqref{L2-lemma-est}, we see that the uniqueness follows. Now, to prove the existence of solution,  for each $k\in \mathbb{N}$, let
\begin{equation}
                \label{eq3.39}
\widehat{Q}_k = (-k^2 , \min\{k^2, T\}) \times B_k^+.
\end{equation}
We consider the equation
\begin{equation} \label{Oge-k.eqn}
x_d^\alpha (a_0 u_t+\lambda c_0 u) - D_i\big(x_d^\alpha (a_{ij} D_j u - F_i)\big)   =  \lambda^{1/2} x_d^\alpha f \quad \text{in} \  \widehat{Q}_k
\end{equation}
with the boundary conditions
\begin{equation} \label{bdr-Oge-k}
u = 0 \quad \text{on} \quad \partial_p  \widehat{Q}_k \setminus \{ x_d =0\} \quad \text{and} \quad
\lim_{x_d \rightarrow 0^+}x_d^\alpha (a_{dj} D_j u - F_d)  = 0.
\end{equation}
where $\partial_p  \widehat{Q}_k $ is the parabolic boundary of $\widehat{Q}_k$. By Galerkin's method, for each $k$, there exists a unique weak solution $u_k \in \cH_2^1(\widehat{Q}_k, \mu)$ to  \eqref{Oge-k.eqn}-\eqref{bdr-Oge-k}. By taking $u_k =0$ on $\Omega_T \setminus  \widehat{Q}_k$, we also have
\[
\begin{split}
 &\sup_{t\in ((-\infty,T)} \|u_{k}(t,\cdot)\|_{L_2(\bR^{d}_+, \mu)} + \|D u_k\|_{L_2(\Omega_T,\mu)}
 +\lambda^{1/2}\|u_k\|_{L_2(\Omega_T,\mu)}\\
& \leq N \|F\|_{L_2(\Omega_T, \mu)} + N\|f\|_{L_2(\Omega_T,\mu)}.
\end{split}
\]
By the weak compactness, there is a subsequence which is still denoted by $\{u_k\}$ and $u\in \cH^1_2(\Omega_T,\mu)$ such that
$$
u_k\rightharpoonup u,\ Du_k\rightharpoonup Du
$$
weakly in $L_2(\Omega_T,\mu)$. By taking the limit in the weak formulation of solutions, it is easily seen that $u$ is a weak solution of  \eqref{eq3.23}. The lemma is proved.
\end{proof}

\section{Equations with simple coefficients} \label{div-sim-sec}

Throughout this section, let $\overline{a}_{ij}: \bR_+ \rightarrow \bR^{d\times d}$ be measurable functions, which satisfy the ellipticity and boundedness conditions: there is a constant $\kappa\in (0,1)$ such that
\begin{equation} \label{ellip-cond}
\kappa |\xi|^2 \leq \overline{a}_{ij}(x_d)\xi_i\xi_j, \quad \text{and} \quad |\overline{a}_{ij}(x_d)|\le \kappa^{-1}, \quad \forall\ \xi\in \bR^d, \quad x_d \in \bR_+.
\end{equation}
Let $\bar{a}_0, \bar{c}_0: \bR_+ \rightarrow \bR$ be measurable functions satisfying
\begin{equation}  \label{a-b.zero}
\kappa \leq \bar{a}_0(x_d), \ \bar{c}_0(x_d) \leq \kappa^{-1} \quad \text{for} \  x_d \in \bR_+.
\end{equation}
We study \eqref{eq3.23} in which the coefficients $a_{ij}$ are replaced with $\overline{a}_{ij}$. More precisely, we consider
\begin{equation} \label{simple-eqn}
\left\{
 \begin{aligned}
x_d^\alpha  (\overline{a}_0(x_d)u_t+\lambda  \overline{c}_0(x_d) u) - D_i\big(x_d^\alpha (\overline{a}_{ij}(x_d) D_j u - F_i)\big)   & =   \sqrt{\lambda} x_d^\alpha f\\
\lim_{x_d \rightarrow 0^+}x_d^\alpha (\overline{a}_{dj}(x_d) D_j u - F_d)  &= 0
\end{aligned} \right.
\end{equation}
in $\Omega_T$. The above equation is slightly different from \eqref{eq3.23} as there are coefficients $\overline{a}_0$ and $\overline{c}_0$ instead of the identity. We do not need this generality for the proofs of our main results for the divergence form equation \eqref{eq3.23}. However, the results below for \eqref{simple-eqn} are needed in the proofs of the main results for the non-divergence form equation \eqref{non-div.eqn} as in \cite{D-P}.

The main result of this section is the following theorem, which is a weak version of Theorem \ref{thm2}.
\begin{theorem} \label{thm3} Let $\alpha \in (-1, \infty)$, $p \in (1,\infty)$, and $\lambda >0$. Suppose that \eqref{ellip-cond} and \eqref{a-b.zero} are satisfied. Then for each $F \in L_p(\Omega_T,\mu)^d$ and $f \in  L_p(\Omega_T,\mu)$, there exists  a unique solution $u\in \cH^1_p(\Omega_T,\mu)$ of \eqref{simple-eqn}. Moreover,
\begin{equation} \label{thm3-est}
\|D u\|_{L_p(\Omega_T,\mu)}+\sqrt{\lambda}\|u\|_{L_p(\Omega_T,\mu)}
\leq N \|F\|_{L_p(\Omega_T, \mu)} + N\|f\|_{L_p(\Omega_T,\mu)},
\end{equation}
where $N = N(d, \alpha, \kappa, p)$.
\end{theorem}
The rest of the section is  devoted to the proof of this theorem. We need some preliminaries to prove it.

\subsection{Lipschitz and Schauder estimates for homogeneous equations}  
Let $\lambda \geq 0$, $z_0 =(t_0, x_0) \in \overline{\bR^{d+1}_+}$ and $r>0$. We study \eqref{simple-eqn} in $Q_r^+(z_0)$ when $F= 0$, $f =0$, i.e., the homogeneous parabolic equation
\begin{equation}
                    \label{eq11.52}
-x_d^\alpha ( \overline{a}_0(x_d) u_t + \lambda \overline{c}_0(x_d) u)+D_i(x_d^\alpha \overline{a}_{ij}(x_d)D_j u)=0
\end{equation}
in $Q_r^+(z_0)$ with the homogeneous conormal boundary condition
\begin{equation}
                    \label{eq12.01}
x_d^\alpha \overline{a}_{dj}(x_d)D_j u=0 \quad \text{if} \quad B_r(x_0) \cap \partial \bR^d_+ \not= \emptyset.
\end{equation}
Our goal is to derive Lipschitz and Schauder estimates for \eqref{eq11.52}-\eqref{eq12.01}. We begin with the following lemma.
\begin{lemma}[Caccioppoli type inequality]
                \label{lem1}
Let $r>0$, $z_0=(t_0,x_0)\in \overline{\bR^{d+1}_+}$, and $u\in \cH^1_{2}(Q^+_r(z_0), \mu)$ be a weak solution to \eqref{eq11.52}-\eqref{eq12.01}. Then we have
$$
\int_{Q^+_{r/2}(z_0)}(|Du|^2 +\lambda|u|^2)\, \mu(dz)
\le Nr^{-2}\int_{Q^+_{r}(z_0)}|u|^2\,\mu(dz)
$$
and
$$
\int_{Q^+_{r/2}(z_0)}|u_t|^2\, \mu(dz)
\le N r^{-2}\int_{Q^+_{r}(z_0)}(|Du|^2 + \lambda |u|^2)\,d\mu (dz),
$$
where $N(d,\alpha,\kappa)>0$.
\end{lemma}
\begin{proof} The proof is more or less standard. For the first inequality, we test the equation with $u\zeta^2$, where $\zeta\in C_0^\infty$ is a smooth function, $\zeta=1$ in $Q_{r/2}(z_0)$, and $\zeta=0$ near the parabolic boundary $\partial_p Q_r(z_0)$. For the second inequality, we test the equation with $u_t\zeta^2$, and then use the fact that $u_t$ satisfies the same equation as $u$ and the first inequality applied to $u_t$. See, for example, the proof of \cite[Lemma 3.3]{DK11}.  We omit the details.
\end{proof}

Next we prove the local boundedness of solutions of \eqref{eq11.52}-\eqref{eq12.01}.
\begin{lemma}[Local boundedness estimate]
                        \label{lem2}
Let $r>0$, $z_0\in \overline{\bR^{d+1}_+}$, and $u\in \cH^1_{2}(Q^+_r(z_0), \mu)$ be a weak solution to \eqref{eq11.52}-\eqref{eq12.01}. Then we have
$$
\|u\|_{L_\infty(Q_{r/2}^+(z_0))}
\le N\Big(\fint_{Q^+_{r}(z_0)}|u(t,x)|^2\, \mu(dz)\Big)^{1/2},
$$
where $N=N(d,\alpha,\kappa)>0$.
\end{lemma}
\begin{proof}
We use the Moser iteration.  For elliptic equations, similar argument was also used in \cite{Sire-1}. By a scaling, we only need to prove the lemma when $r=1$. For each $R, \rho \in (0, 1]$ with $\rho < R$, let $\phi \in C_0^\infty((t_0-R^2,t_0+R^2)\times B_R(x_0))$ be a cut-off function satisfying
\[
\phi =1 \ \text{in} \ Q_\rho(z_0), \quad 0\leq \phi \leq 1,
\quad  \text{and} \  |D \phi|^2 + |\partial_t \phi| \leq \frac{N(d)}{(R-\rho)^2} \quad \text{in} \  Q_R(z_0).
\]
Let $w = u_+$. For $\beta \geq 2$,  using $\phi^2 w^{\beta-1}$ as a test function for the equation \eqref{eq11.52} and using \eqref{ellip-cond}, we obtain
\[
\begin{split}
& \frac{d}{dt} \int_{B_R^+(x_0)} \bar a_0(x_d)w^\beta \phi^2 \,\mu(dx) + \frac{4\kappa (\beta-1)}{\beta}\int_{B_R^+(x_0)} |D(w^{\beta/2})|^2 \phi^2 \,\mu(dx) \\
& \leq 2  \beta \int_{B_R^+(x_0)} \bar a_0 (x_d) w^{\beta} \phi |\phi_t| \,\mu(dx) +4 d\kappa^{-1} \int_{B_R^+(x_0)} |D (w^{\beta/2})| |D\phi| \phi w^{\beta/2} \,\mu(dx),
\end{split}
\]
where we used the fact that $\lambda \bar c_0(x_d) u \phi^2 w^{\beta-1} \geq 0$. As $\beta \geq 2$, we have $\frac{\beta-1}{\beta} \geq \frac{1}{2}$. It then follows that
\[
\begin{split}
& \frac{d}{dt} \int_{B_R^+(x_0)} \bar a_0(x_d)w^\beta \phi^2 \,\mu(dx) + 2\kappa \int_{B_R^+(x_0)} |D(w^{\beta/2})|^2 \phi^2 \,\mu(dx) \\
& \leq 2 \beta \int_{B_R^+(x_0)} \bar a_0(x_d) w^{\beta} \phi |\phi_t| \,\mu(dx) + 4d\kappa^{-1} \int_{B_R^+(x_0)} |D w^{\beta/2}| |D\phi| \phi w^{\beta/2} \,\mu(dx).
\end{split}
\]
By applying Young's inequality to the last term and then cancelling similar terms, we have
\[
\begin{split}
& \frac{d}{dt} \int_{B_R^+(x_0)} \bar a_0(x_d)w^\beta \phi^2 \,\mu(dx)
+ \int_{B_R^+(x_0)} |D(w^{\beta/2} \phi)|^2 \,\mu(dx) \\
& \leq N  \beta \int_{B_R^+(x_0)} {w}^{\beta} \big( |\phi_t| +|D\phi|^2\big)\,\mu(dx),
\end{split}
\]
where $N=N(d,\kappa)$ and we used \eqref{a-b.zero}. Integrating this  estimate with respect to $t$ on $(t_0-R^2, t_0)$ and using \eqref{a-b.zero} again, we find that
\[
\begin{split}
& \sup_{t \in (t_0-R^2, t_0)}\fint_{B_R^+(x_0)} w^{\beta} \phi^2 \,\mu(dx) + R^2 \fint_{Q_R^+(z_0)} |D(w^{\beta/2} \phi)|^2 \,\mu(dz) \\
 & \leq \frac{N(d, \kappa) \beta }{(R-\rho)^2}\fint_{Q_R^+(z_0)}w^{\beta}  \,\mu(dz).
 \end{split}
\]
From this estimate and Lemma \ref{Sobolev-imbed},  it follows that
\begin{equation} \label{Moser-46}
\bigg( \fint_{Q_\rho^+(z_0)} w^{\beta l_0} \,\mu(dz) \bigg)^{\frac{1}{\beta l_0}} \leq \bigg( \frac{N}{R-\rho}\bigg)^{\frac{2}{\beta}} \beta^{\frac 1 \beta} \bigg( \fint_{Q_R^+(z_0)} w^{\beta}\,\mu(dz) \bigg)^{\frac{1}{\beta}}.
\end{equation}
We now choose a sequence of radii
\[
r_0 = 1, \quad r_{k+1} = \frac{r_k + {1}/{2}}{2},
\]
and  a sequence of exponents
\[
\beta_0 = 2, \ \beta_{k+1} =  \beta_k l_0 ,  \quad  k = 0, 1, 2, \ldots,
\]
such that
\[
\lim_{k\rightarrow \infty} r_k = \frac{1}{2}, \quad \lim_{k\rightarrow \infty} \beta_k = \infty, \quad \text{and} \quad r_{k} - r_{k+1} = \frac{1}{2^{k+2}}, \quad k = 0, 1,2,\ldots
\]
By applying \eqref{Moser-46} with $R = r_k$, $\rho = r_{k+1} < R$, and $\beta = \beta_k$, we have
\[
\bigg( \fint_{Q_{r_{k+1}(z_0)}^+} w^{\beta_{k+1}} \,\mu(dz) \bigg)^{\frac{1}{\beta_{k+1}}} \leq \big(4N\big)^{\frac{2}{\beta_k}} 2^{\frac{2k}{\beta_k}} \beta_k^{\frac{1}{\beta_k}} \bigg( \fint_{Q_{r_k}(z_0)} w^{\beta_k}\,\mu(dz) \bigg)^{\frac{1}{\beta_k}}.
\]
By iterating this estimate, we obtain
\begin{equation} \label{k-Moser}
\bigg( \fint_{Q_{r_{k+1}}^+(z_0)} w^{\beta_{k+1}} \,\mu(dz) \bigg)^{\frac{1}{\beta_{k+1}}} \leq M_k \bigg( \fint_{Q_1^+(z_0)} w^{2}\,\mu(dz) \bigg)^{\frac{1}{2}},
\end{equation}
where
\[
M_k = (4N)^{\sum_{j=0}^k{2}/{\beta_j}}  2^{\sum_{j=0}^k{2j}/{\beta_j}} \prod_{j=0}^k\beta_j^{\frac{1}{\beta_j}}.
\]
As
\[ \sum_{j=0}^{\infty} 2/\beta_j <\infty, \quad
\sum_{j=0}^{\infty} 2 j/\beta_j <\infty, \quad \text{and} \quad
 \prod_{j=0}^\infty\beta_j^{\frac{1}{\beta_j}}<\infty,\]
we conclude that $\{M_k\}_k$ is convergent. Therefore, by sending $k \rightarrow \infty$, we deduce from \eqref{k-Moser} that
\[
\| u_+\|_{L_\infty(Q_{1/2}^+(z_0))} \leq N \bigg( \fint_{Q_1^+(z_0)} u_+^{2}(t,x)\,\mu(dz) \bigg)^{\frac{1}{2}}.
\]
With the same argument, we can get a similar estimate for $u_{-} = \max\{-u,0\}$. Hence,
\[
\| u\|_{L_\infty(Q_{1/2}^+(z_0))} \leq N \bigg( \fint_{Q_1^+(z_0)} |u(t,x)|^{2}\,\mu(dz) \bigg)^{\frac{1}{2}}.
\]
The lemma is proved.
\end{proof}

We recall that for $\beta \in (0, 1]$ and each parabolic cylinder $Q \subset \bR^{d+1}$, the $\beta$-H\"{o}lder semi-norm of a function $f$ in $Q$ is defined as
\[
[f]_{C^{\beta/2, \beta}(Q)} = \sup_{\substack{(t,x), (s,y) \in Q\\ (t,x) \not=(s,y)}} \frac{|f(t,x)-f(s,y)|}{|t-s|^{\beta/2} + |x-y|^{\beta}}.
\]
The following proposition is the key step of the proof.
\begin{proposition}  \label{prop1} Let $q \in (1, 2]$, $r>0$, $z_0\in \overline{\bR^{d+1}_+}$, and $u\in \cH^1_{2}(Q^+_r(z_0), \mu)$ be a weak solution to \eqref{eq11.52}-\eqref{eq12.01}. Then we have
\begin{equation}
                                \label{eq5.51}
                                \begin{split}
&\|Du\|_{L_\infty(Q_{r/2}^+(z_0))} +\sqrt \lambda \|u\|_{L_\infty(Q_{r/2}^+(z_0))}\\
& \le N\Big(\fint_{Q^+_{r}(z_0)}\big(|Du|^q  + \lambda^{q/2} |u|^q \big)\,\mu(dz)\Big)^{1/q}
\end{split}
\end{equation}
and
\begin{equation}
                                        \label{eq5.52}
                                        \begin{split}
& [D_{x'}u]_{C^{1/2,1}(Q_{r/2}^+(z_0))}+[\mathcal{U}]_{C^{1/2,1}(Q_{r/2}^+(z_0))}
+\sqrt\lambda[u]_{C^{1/2,1}(Q_{r/2}^+(z_0))}\\
& \le Nr^{-1}\Big(\fint_{Q^+_{r}(z_0)}\big(|Du|^{q}  +  \lambda^{q/2} |u|^q\big) \,\mu(dz)\Big)^{1/q},
\end{split}
\end{equation}
where $\mathcal{U}=\overline{a}_{dj}(x_d)D_j u$ and $N = N(d, \alpha,\kappa, q)$.
\end{proposition}
\begin{proof}  First of all, whenever the lemma is proved for $q =2$, the case $q \in (1,2)$ follows by a standard iteration. See, for example, \cite[pp. 80--82]{Giaq}. Therefore, we only consider the case when $q=2$. As before, we may assume that $r=1$. The bound of $\|u\|_{L_\infty(Q_{r/2}^+(z_0))}$ follows from Lemma \ref{lem2}. Since $D_{x'}u$ and $u_t$ satisfy the same equation as $u$, from Lemmas \ref{lem2} again we have
\begin{equation*}
\|D_{x'}u\|_{L_\infty(Q_{1/2}^+(z_0))}
\le N\Big(\fint_{Q^+_{2/3}(z_0)}|D_{x'}u|^2\,\mu(dz)\Big)^{1/2}
\end{equation*}
and
\begin{equation*}
\|u_t\|_{L_\infty(Q_{1/2}^+(z_0))}
\le N\Big(\fint_{Q^+_{2/3}(z_0)}|u_t|^2\,\mu(dz)\Big)^{1/2}.
\end{equation*}
To make this rigorous, we need to use the finite-difference quotient and pass to the limit. These together with Lemma \ref{lem1} give
\begin{equation}
                            \label{eq1.41}
\begin{split}
& \|D_{x'}u\|_{L_\infty(Q_{1/2}^+(z_0))}
+\|u_t\|_{L_\infty(Q_{1/2}^+(z_0))}\\
& \le N\Big(\fint_{Q^+_{1}(z_0)}\big(|Du|^2 + \lambda |u|^2\big)\,\mu(dz)\Big)^{1/2}.
\end{split}
\end{equation}
Moreover, again from Lemma \ref{lem1}, we also have for any $i,j=0,1,2,\ldots$  satisfying $i+j\ge 1$,
\begin{equation}
                        \label{eq1.16}
\int_{Q^+_{1/2}(z_0)}\Big( |\partial_t^i D^j_{x'}u|^2+|\partial_t^i D^j_{x'}Du|^2\,\Big)\,\mu(dz)\le N\int_{Q^+_{1}(z_0)}\big(|Du|^2 +\lambda|u|^2\big)\,\mu(dz),
\end{equation}
where $N=N(d,\kappa,i,j)$.

Next we estimate $D_d u$. We first consider the boundary estimate and, without loss of generality, we take $z_0=0$. We use a bootstrap argument. Since $\mathcal{U}=\overline{a}_{dj}(x_d)D_j u$, from the equation we have
$$
D_d(x_d^\alpha \mathcal{U})=x_d^\alpha \big(\bar a_0(x_d)u_t + \lambda \bar c_0(x_d) u-\sum_{i=1}^{d-1}D_i(\bar a_{ij}D_ju) \big).
$$
By using the boundary condition and  H\"{o}lder's inequality, we get for any $z\in Q_1^+$,
\begin{align}
                                        \label{eq1.24}
&x_d^\alpha |\mathcal{U}|\le N\int_0^{x_d}s^\alpha (|u_t(z',s)| + \lambda |u(z',s)|+|DD_{x'}u(z',s)|)\,ds\\
&\le N\Big(\int_0^{x_d}s^\alpha (|u_t(z',s)|^2 + \lambda^2 |u(z',s)|^2 +|DD_{x'}u(z',s)|^2)\,ds\Big)^{\frac 1 2}
\Big(\int_0^{x_d}s^\alpha \,ds\Big)^{\frac 1 2}\nonumber.
\end{align}
Thus, when $x_d\in (0, 1/2]$, by the Sobolev embedding in the $z'$ variables, \eqref{eq1.16}, and Lemma \ref{lem1},  for an integer $k\ge (d+1)/4$,
\begin{align*}
&x_d^\alpha |\mathcal{U}|\\
&\le N\Big(\int_0^{1/2}s^\alpha (|u_t(z',s)|^2 + \lambda^2 |u(z',s)|^2 +|DD_{x'}u(z',s)|^2)\,ds\Big)^{\frac 1 2}
\Big(\int_0^{x_d}s^\alpha \,ds\Big)^{\frac 1 2}\\
& \le N\Big(\int_0^{1/2}s^\alpha (\|u_t(\cdot,s)\|_{W^{k,2k}_2(Q'_{1/2})}^2 + \lambda^2 \|u(\cdot,s)\|_{W^{k,2k}_2(Q'_{1/2})}^2 \\ &
\qquad  +\|DD_{x'}u(\cdot,s)\|_{W^{k,2k}_2(Q'_{1/2})}^2)\,ds\Big)^{\frac 1 2}
\Big(\int_0^{x_d}s^\alpha \,ds\Big)^{\frac 1 2}\\
&\le N\Big(\int_{Q_1^+}(|Du|^2 + \lambda |u|^2)\,\mu(dz)\Big)^{\frac 1 2}
x_d^{(\alpha+1)/2},
\end{align*}
which implies that
\begin{equation}
                    \label{eq2.12}
|\mathcal{U}|\le N\Big(\int_{Q_1^+}\big(|Du|^2  + \lambda |u|^2\big) \,\mu(dz)\Big)^{\frac 1 2}
x_d^{(1-\alpha)/2}\quad\text{in}\ Q_{1/2}^+.
\end{equation}
This together with \eqref{eq1.41} gives
$$
|Du|\le N\Big(\int_{Q_1^+}\big(|Du|^2  + \lambda |u|^2\big)\, \mu(dz)\Big)^{\frac 1 2}
x_d^{-(1-\alpha)_-/2}\quad\text{in}\ Q_{1/2}^+.
$$
Since $D_{x'} u$ satisfies the same equation,  by a covering argument and Lemma  \ref{lem1} we have
\begin{align}
                                        \label{eq1.48}
|DD_{x'}u|&\le N\Big(\int_{Q_{2/3}^+}\big(|DD_{x'}u|^2  {+ \lambda |D_{x'}u|^2} \big)\, \mu(dz)\Big)^{\frac 1 2}x_d^{-(1-\alpha)_-/2}\nonumber\\
&\le N\Big(\int_{Q_1^+} |D_{x'}u|^2)\, \mu(dz)\Big)^{\frac 1 2}
x_d^{-(1-\alpha)_-/2}\quad\text{in}\ Q_{1/2}^+.
\end{align}
Now we plug \eqref{eq1.41} and \eqref{eq1.48} into \eqref{eq1.24} and use Lemmas \ref{lem2} and \ref{lem1} to get
\begin{align*}
|\mathcal{U}|&\le Nx_d^{-\alpha} \int_0^{x_d}s^\alpha s^{-(1-\alpha)_-/2}\,ds
\Big(\int_{Q_1^+}\big( |Du|^2 {+ \lambda |u|^2} \big)\, \mu(dz)\Big)^{\frac 1 2}\\
&\le Nx_d^{1-(1-\alpha)_-/2}\Big(\int_{Q_1^+}\big(|Du|^2  {+ \lambda |u|^2} \big) \, \mu(dz)\Big)^{\frac 1 2}\quad\text{in} \ Q_{1/2}^+,
\end{align*}
which improves \eqref{eq2.12}.
Repeating this procedure, in finite many steps, we get
\begin{equation}
                        \label{eq7.48}
|\mathcal{U}|\le Nx_d\Big(\int_{Q_1^+}\big(|Du|^2  {+ \lambda |u|^2} \big)\, \mu(dz)\Big)^{\frac 1 2}
\end{equation}
and therefore
$$
|Du|\le N\Big(\int_{Q_1^+}\big(|Du|^2  {+ \lambda |u|^2} \big)\, \mu(dz)\Big)^{\frac 1 2}
\quad\text{in}\  Q_{1/2}^+,
$$
which gives \eqref{eq5.51} in this case.

In the interior case when $x_{0d}\ge 2r=2$, the coefficients ${\tilde{a}_{ij}(x_d)}=x_d^\alpha \bar a_{ij}(x_d)$ are nondegenerate in $Q_{2/3}(z_0)$ and independent of $z'$. By using the standard energy estimate (cf. \cite[Lemma 3.5]{DK11}), we also have
\begin{equation}
                        \label{eq11.31}
|Du|\le N\Big(\int_{Q_{1}(z_0)}\big(|Du|^2  {+ \lambda |u|^2} \big)\, dz\Big)^{\frac 1 2}
\quad\text{in}\ Q_{1/2}(z_0).
\end{equation}
Since in $Q_{1}(z_0)$, $x_d\sim x_{0d}$ so that $ \mu(dz)\sim x_{0d}^\alpha \, dz$, we also obtain \eqref{eq5.51} in the interior case.  Moreover, \eqref{eq1.41} still holds in this case. When $x_{0d}\in (0,2)$, \eqref{eq5.51} follows from a covering argument and the doubling property of $\mu$.

It remains to prove \eqref{eq5.52}.  By using \eqref{eq5.51} and \eqref{eq1.41}, we obtain the bound of the third term on the left-hand side of \eqref{eq5.52}. Since $D_{x'}u$ and $u_t$ satisfy the same equation as $u$, from  \eqref{eq5.51}, \eqref{eq1.16}, and Lemma \ref{lem1},
we have
\begin{align}
            \label{eq8.06}
&\|DD_{x'}u\|_{L_\infty(Q_{1/2}^+(z_0))}+\|Du_t\|_{L_\infty(Q_{1/2}^+(z_0))}\nonumber\\
&\le N\Big(\fint_{Q^+_{2/3}(z_0)} \big(|DD_{x'}u|^2+\lambda |D_{x'}u|^2
+|Du_t|^2+\lambda |u_t|^2\big)
\,\mu(dz)\Big)^{1/2} \nonumber\\
&\le N\Big(\fint_{Q^+_{1}(z_0)}\big(|Du|^2  + \lambda |u|^2\big)\,\mu(dz)\Big)^{1/2},
\end{align}
which yields
\begin{align*}
&[D_{x'}u]_{C^{1/2,1}(Q_{1/2}^+(z_0))}+\| \cU_t\|_{L_\infty(Q_{1/2}^+(z_0))}
+\|D_{x'}\cU\|_{L_\infty(Q_{1/2}^+(z_0))}\\
&\le N\Big(\fint_{Q^+_{1}(z_0)}\big(|Du|^2  + \lambda |u|^2 \big)\,\mu(dz)\Big)^{1/2}.
\end{align*}

To estimate $D_d  \cU$, we again discuss two cases. In the boundary case when $z_0=0$, from the equation we have
\begin{equation}
                                \label{eq8.12}
D_d \mathcal{U}= \overline{a}_0u_t +\lambda \overline{c}_0 u-\sum_{i=1}^{d-1}\bar a_{ij}D_{ij} u-\alpha x_d^{-1} \cU,
\end{equation}
which together with \eqref{eq1.41}, \eqref{eq7.48}, \eqref{eq8.06}, and Lemma \ref{lem2} gives
\begin{align}
                                \label{eq8.16}
\|D_{d}\mathcal{U}\|_{L_\infty(Q_{1/2}^+(z_0))}\le N\Big(\fint_{Q^+_{1}(z_0)}\big(|Du|^2  + \lambda |u|^2\big)\,\mu(dz)\Big)^{1/2}.
\end{align}
In the interior case when $x_{0d}\ge 2$, by \eqref{eq8.12}, \eqref{eq1.41}, \eqref{eq11.31}, \eqref{eq8.06}, and  Lemma \ref{lem2}, we still get \eqref{eq8.16}. This completes the proof of \eqref{eq5.52} and thus the proposition.
\end{proof}

From Lemma \ref{L-2-lemma} and Proposition \ref{prop1}, we obtain the following solution decomposition.
\begin{proposition} \label{Simple-approx} Let $z_0\in \overline{\Omega_T}$ and $r >0$. Suppose that $F \in L_{2}(Q_{2r}^+(z_0), \mu)^d$, $f \in L_{2}(Q_{2r}^+(z_0), \mu)$, and $u \in \cH^{1}_2(Q_{2r}^+(z_0),\mu)$ is a weak solution of \eqref{simple-eqn} in $Q_{2r}^+(z_0)$. Then we can write
\[
u(t, x) = v(t, x) + w(t, x) \quad  \text{in}\,\, Q_{2r}^+(z_0),
\]
where $v$ and $w$ are functions in $\cH_2^1(Q_{2r}^+(z_0), \mu)$ and satisfy
\begin{equation}  \label{0506-tU-est}
\fint_{Q_{2r}^+(z_0)}|V|^2  \,{\mu}(dz)  \leq N \fint_{Q_{2r}^+(z_0)}\Big( |F|^2 +  |f|^2) \,{\mu}(dz)
\end{equation}
and
\begin{align}\label{0506-W.est}
\|W\|_{L_\infty(Q_{r}^+(z_0))}^{2} 
\leq N  \fint_{Q_{2r}^+(z_0)} |U|^2 \,{\mu}(dz) + N \fint_{Q_{2r}^+(z_0)} \Big(|F|^2 + |f|^2\Big) \,{\mu}(dz),
\end{align}
 where $N = N(d,\kappa, \alpha)$ and
$$
V=|Dv|+\lambda^{1/2}|v|,\quad W=|Dw|+\lambda^{1/2}|w|,\quad U=|D u|+\lambda^{1/2}|u|.
$$
\end{proposition}
\begin{proof}  Let $v \in \cH^1_2(\Omega_T, \mu)$ be a weak solution of the equation
\begin{equation*}
\begin{split}
& x_d^\alpha  ( {\overline{a}_0(x_d)}v_t  +\lambda {\overline{c}_0(x_d)} v)  - D_i
\big(x_d^\alpha(\overline{a}_{ij}(x_d) D_j v - F_{i}(z)\chi_{Q_{2r}^+(z_0)}(z))\big)  \\
& =  \lambda^{1/2} x_d^\alpha f (z) \chi_{Q_{2r}^+(z_0)}(z)  \quad \text{in} \  \Omega_T
\end{split}
\end{equation*}
with the boundary condition
\begin{equation*}
\lim_{x_d \rightarrow 0^+} x_d^\alpha (\overline{a}_{dj}(x_d) D_j v - F_d(z)\chi_{Q_{2r}^+(z_0)}(z)) =  0.
\end{equation*}
Then \eqref{0506-tU-est} follows from Lemma \ref{L-2-lemma}. Now let $w = u - v$ so that $w \in \cH_2^1(Q_{2r}^+(z_0))$ is a weak solution of
\[
x_d^\alpha  (\overline{a}_0(x_d) w_t +\lambda \overline{c}_0(x_d) w)   - D_i \big(x_d^\alpha \overline{a}_{ij}(x_d) D_j  w\big)   = 0 \quad \text{in} \  Q_{2r}^+(z_0)
\]
with the boundary condition
\[
\lim_{x_d\rightarrow 0^+} x_d^\alpha \overline{a}_{dj}(x_d) D_j  w = 0 \quad \text{if} \quad B_{2r}(x_0) \cap \partial \overline{\bR}^{d}_+ \not= \emptyset.
\]
By Proposition \ref{prop1} and the triangle inequality, we get \eqref{0506-W.est}.
The proof of the proposition is completed.
\end{proof}
\subsection{Proof of Theorem \ref{thm3}}
We are now ready to give the proof of Theorem \ref{thm3}.
\begin{proof}
When $p =2$, Theorem \ref{thm3} follows from Lemma \ref{L-2-lemma}. Therefore, we only need to consider the cases when $p \in (2,\infty)$ and $p \in (1,2)$.\\

\noindent
{\bf Case I}: $p \in (2,\infty)$. Let $u \in \cH_{2,\textup{loc}}^1(\Omega_T, \mu)$ be a weak solution of \eqref{simple-eqn}.  It follows from Proposition \ref{Simple-approx} that for every $z_0 \in \overline{\Omega}_T$ and $r >0$, we have the decomposition
\[
u =v+ w \quad \text{in} \  Q_{2r}^+(z_0),
\]
where $v$ and $w$ satisfy \eqref{0506-tU-est} and \eqref{0506-W.est}. Then \eqref{thm3-est} follows from the standard real variable argument. See, for example, \cite{DK11b}. We omit the details.

By \eqref{thm3-est}, the uniqueness of solutions follows. Hence, it remains to prove the existence of the solution. Recall the definition \eqref{eq3.39}. For $k=1,2,\ldots$, let $F^{(k)}=F(z) \chi_{\widehat Q_k}(z)$. Then $F^{(k)}\in L_2(\Omega_T,\mu)^{d}\cap L_p(\Omega_T,\mu)^{d}$ and by the dominated convergence theorem, $F^{(k)}\to F$ in $L_p(\Omega_T,\mu)$ as $k\to \infty$. Similarly, we define $\{f^{(k)}\}\subset L_2(\Omega_T,\mu)\cap L_p(\Omega_T,\mu)$. Let $u^{(k)} {\in \cH_2^1(\Omega_T, \mu)}$ be the weak solution of the equation \eqref{simple-eqn} with $F^{(k)}$ and $ f^{(k)}$ in place of $F$ and $f$, respectively. The existence of $u^{(k)}$ follows from Lemma \ref{L-2-lemma}. By the estimate \eqref{thm3-est}, we have $u^{(k)}\in \cH^{1}_p(\Omega_T,\mu)$. Moreover, by the strong convergence of $\{F^{(k)}\}$ and $\{f^{(k)}\}$ in $L_{p}(\Omega_T,\mu)$, we infer that
$\{u^{(k)}\}$ is a Cauchy sequence in $\cH^{1}_{p}(\Omega_T,\mu)$. Let $u\in \cH^{1}_{p}(\Omega_T,\mu)$ be its limit. Then, by passing to the limit in the weak formulation of solutions, it is easily seen that $u$ is a solution to the equation \eqref{simple-eqn}.\\

\noindent
{\bf Case II}: $p \in (1, 2)$. We use a duality argument. We first prove the estimate \eqref{thm3-est}. Let $q = {p}/(p-1) \in (2,\infty)$ and let $G \in L_q(\Omega_T, \mu)^d$ and $g \in L_q(\Omega_T, \mu)$. We consider the adjoint problem in $\bR \times \bR^d_+$
\begin{equation} \label{adj-eqn-sim}
\left\{
\begin{aligned}
x_d^\alpha(- \bar a_0 v_t + \lambda \bar c_0 v) - D_i\big(x_d^\alpha(\overline{a}_{ji}(x_d) D_{j} v - G_{i}\chi_{(-\infty, T)})\big) & =  \lambda^{1/2} x_d^\alpha g\chi_{(-\infty, T)},\\
\lim_{x_d\rightarrow 0^+}x_d^\alpha(\overline{a}_{jd} D_{j} v - G_d\chi_{(-\infty, T)}) & = 0.
\end{aligned} \right.
\end{equation}
By {\bf Case I}, there exists unique solution $v \in \cH^1_q(\bR \times \bR_+^d,  \mu)$ of the above equation, which satisfies
\begin{equation}
                            \label{eq10.35}
\int_{\bR \times \bR^d_+} \big(|Dv|^q + \lambda^{q/2} |v|^q \big) \,\mu(dz) \leq N\int_{\Omega_T} \big(|G|^q + |g|^q \big) \,\mu(dz).
\end{equation}
Moreover, by the uniqueness of solutions, we have $v =0$ for $t \geq T$. It follows from the equations \eqref{simple-eqn} and \eqref{adj-eqn-sim} that
\[
\int_{\Omega_T}\big(G\cdot \nabla u  + \lambda^{1/2} g u \big)\,\mu(dz)
= \int_{\Omega_T}\big(F\cdot \nabla v + \lambda^{1/2} f v \big)\,\mu(dz).
\]
Therefore,  by H\"older's inequality and \eqref{eq10.35},
\[
\begin{split}
& \left|\int_{\Omega_T}\big(G\cdot \nabla u + \lambda^{1/2} g u \big)\,\mu(dz)\right| \\
& \leq  \|F\|_{L_p(\Omega,\mu)} \|\nabla v\|_{L_q(\Omega_T, \mu)} + \lambda^{1/2} \|f\|_{L_{p}(\Omega_T, \mu)} \| v\|_{L_q(\Omega_T, \mu)}\\
& \leq N\Big(\|F\|_{L_p(\Omega,\mu)} + \|f\|_{L_{p}(\Omega_T, \mu)}\Big)
\Big( \|G\|_{L_q(\Omega_T, \mu)} + \| g\|_{L_q(\Omega_T, \mu)} \Big).
\end{split}
\]
From this last estimate and as $G$ and $g$ are arbitrary, we obtain \eqref{thm3-est}.

It now remains to prove the existence of solution $u \in \cH_p^1(\Omega_T, \mu)$. We proceed slightly differently  from {\bf Case I} and follow the argument in \cite[Section 8]{MR3812104}. For $i=1,2,\ldots, d$ and $k=1,2,\ldots$, let
$$
F^{(k)}_i=\max(-k,\min(k,F_i))\chi_{\widehat Q_k}.
$$
Then $F^{(k)}\in L_2(\Omega_T,\mu)^{d}\cap L_p(\Omega_T,\mu)^{d}$ and by the dominated convergence theorem, $F^{(k)}\to F$ in $L_p(\Omega_T,\mu)$ as $k\to \infty$. Similarly, we define $\{f^{(k)}\}\subset L_2(\Omega_T,\mu)\cap L_p(\Omega_T,\mu)$. By Lemma \ref{L-2-lemma}, there is a unique weak solution $u^{(k)}\in \cH^{1}_{2}(\Omega_T,\mu)$ to the equation \eqref{simple-eqn} with $F^{(k)}$ and $f^{(k)}$ in place of $F$ and $f$, respectively. As in {\bf Case I}, it suffices to prove that  $u^{(k)} \in \cH_p^1(\Omega_T,  \mu)$. Let us fix a $k \in \mathbb{N}$. Because ${\mu}$ is a doubling measure, there exists $N_0 = N_0(\alpha, d)>0$  such that
 \begin{equation}  \label{eq7.36}
{\mu}(\widehat Q_{2r})\le N_0{\mu}(\widehat Q_{r}), \quad \forall \ r >0.
\end{equation}
Since $u^{(k)}\in \cH^{1}_{2}(\Omega_T,\mu)$, by H\"older's inequality,
\begin{equation} \label{eq7.38}
 \|u^{(k)}\|_{L_p(\widehat Q_{2k},\mu)}+\|Du^{(k)}\|_{L_p(\widehat Q_{2 k} ,\mu)}<\infty.
\end{equation}
Therefore, it remains to prove that $\|u^{(k)}|\|_{L_p(\Omega_T\setminus \widehat{Q}_{2k})} <\infty$.  To this end, for $j\ge 0$,  let $\eta_j$ be such that
\begin{align*}
\eta_j&\equiv 0\quad \text{in}\,\, \widehat Q_{2^jk},  \quad \eta_j \equiv 1 \quad \text{outside}\,\, \widehat Q_{2^{j+1}k},
\end{align*}
and $|D \eta_j|\le C_0 2^{-j}, \quad |(\eta_j)_t|\le C_0 2^{-2j}$, where $C_0$ is independent of $j$.  Observe that the supports of $F^{(k)}$ and $f^{(k)}$ are in $\widehat{Q}_{{k}}$, while the supports of $\eta_j$ are all outside $\widehat{Q}_{{k}}$.  Consequently, $\eta_j F_i^{(k)} \equiv \eta_j f^{(k)}  \equiv F_i^{(k)} D_i\eta_j \equiv 0$ for every $i = 1, 2,\ldots, d$ and $j =0,1,\ldots$. Because of this, a simple calculation reveals that $w^{(k,l)}:=u^{(k)}\eta_l\in \cH^{1}_{2}(\Omega_T,\mu)$ is a weak solution of
\begin{equation*}
\left\{
\begin{aligned}
x_d^\alpha \big( \overline{a}_0 w^{(k,l)}_t  +\lambda  \overline{c}_0 w^{(k,l)}\big) - D_i\big(x_d^\alpha (\overline{a}_{ij} D_j w^{(k,l)} - F^{(k,l)}_i)\big)  & =  \lambda^{1/2} x_d^\alpha f^{(k,l)} \\
\lim_{x_d \rightarrow 0^+}x_d^\alpha(\overline{a}_{dj} D_j w^{(k,l)} - F^{(k,l)}_d) & =  0
\end{aligned} \right.
\end{equation*}
in $\Omega_T$, where
\[
\begin{split}
F^{(k,l)}_i & =  u^{(k)}\overline{a}_{ij}D_j \eta_l, \quad i = 1, 2,\ldots, d,\\
 f^{(k,l)} & = \lambda^{-1/2}\big(u^{(k)}(\eta_l)_t - \overline{a}_{ij} D_j u^{(k)} D_i\eta_l \big).
\end{split}
\]
Now, by applying the estimate \eqref{L2-lemma-est} to the above equation of $w^{(k,l)}$, we have
\begin{align*}
& \|Dw^{(k,l)} \|_{L_2(\Omega_T,\mu)}+ \sqrt{\lambda} \|w^{(k,l)}\|_{L_2(\Omega_T,\mu)}\\
& \le N \|F^{(k,j)}\|_{L_2(\Omega_T,\mu)}+N
\|f^{(k,l)}\|_{L_2(\Omega_T,\mu)} ,
\end{align*}
which implies that
\begin{align*}
&\|Du^{(k)}\|_{L_2(\widehat Q_{2^{j+2}k}\setminus \widehat Q_{2^{j+1}k},\mu)}+\sqrt{\lambda} \|u^{(k)}\|_{L_2(\widehat Q_{2^{j+2}k}\setminus \widehat Q_{2^{j+1}k},\mu)}\\
&\le
N\Big( 2^{-j}\|u^{(k)}\|_{L_2(\widehat Q_{2^{j+1}k}\setminus \widehat Q_{2^{j}k},\mu)}+ \lambda^{-1/2}2^{-2j}
\|u^{(k)}\|_{L_2(\widehat Q_{2^{j+1}k}\setminus \widehat Q_{2^{j}k},\mu)}\\
&\quad + \lambda^{-1/2}2^{-j}\|Du^{(k)}\|_{L_2(\widehat Q_{2^{j+1}k}\setminus \widehat Q_{2^{j}k},\mu)} \Big)\\
&\le C2^{-j}\big(\|Du^{(k)}\|_{L_2(\widehat Q_{2^{j+1}k}\setminus \widehat Q_{2^{j}k},\mu)}+\sqrt{\lambda} \|u^{(k)}\|_{L_2(\widehat Q_{2^{j+1}k}\setminus \widehat Q_{2^{j}k},\mu)}\big)
\end{align*}
for every $j \geq 1$, where $C$ also depends on $\lambda$, but is independent of $j$. By iterating the last estimate, we obtain
\begin{align}  \label{eq9.01}
  &\|Du^{(k)}\|_{L_2(\widehat Q_{2^{j+1}k}\setminus \widehat Q_{2^{j}k},\mu)}+\sqrt{\lambda} \|u^{(k)}\|_{L_2(\widehat Q_{2^{j+1}k}\setminus \widehat Q_{2^{j}k},\mu)}\nonumber\\
&\le C^j2^{-j(j-1)/2}\big(\|Du^{(k)}\|_{L_2(\widehat Q_{2k},\mu)}+ \sqrt{\lambda} \|u^{(k)}\|_{L_2(\widehat Q_{2k},\mu)}\big).
\end{align}
Finally, by H\"older's inequality, \eqref{eq7.36}, and \eqref{eq9.01}, we have
\begin{align*}
&\|Du^{(k)}\|_{L_p(\widehat Q_{2^{j+1}k}\setminus \widehat Q_{2^{j}k},\mu)}+\sqrt{\lambda} \|u^{(k)}\|_{L_p(\widehat Q_{2^{j+1}k}\setminus \widehat Q_{2^{j}k},\mu)}\\
&\le ({\mu}(\widehat Q_{2^{j+1}k}))^{\frac 1 p-\frac 1 2}\big(\|Du^{(k)}\|_{L_2(\widehat Q_{2^{j+1}k}\setminus \widehat Q_{2^{j}k},\mu)}+ \sqrt{\lambda} \|u^{(k)}\|_{L_2(\widehat Q_{2^{j+1}k}\setminus \widehat Q_{2^{j}k},\mu)}\big)\\
&\le N_0^{j(\frac 1 p-\frac 1 2)}({\mu}(\widehat Q_{2k}))^{\frac{1}{p}-\frac 1 2}C^j2^{-\frac {j(j-1)} 2}\big(\|Du^{(k)}\|_{L_2(\widehat Q_{2k},\mu)}+\sqrt{\lambda} \|u^{(k)}\|_{L_2(\widehat Q_{2k},\mu)}\big).
\end{align*}
Hence,
\[
\begin{split}
&\|Du^{(k)}\|_{L_p(\Omega_T \setminus \widehat Q_{2k},\mu)}+\sqrt{\lambda} \|u^{(k)}\|_{L_p(\Omega_T \setminus \widehat Q_{2k},\mu)}
\\
&=\sum_{j=1}^\infty \Big( \|Du^{(k)}\|_{L_p(\widehat Q_{2^{j+1}k}\setminus \widehat Q_{2^{j}k},\mu)}+\sqrt{\lambda} \|u^{(k)}\|_{L_p(\widehat Q_{2^{j+1}k}\setminus \widehat Q_{2^{j}k},\mu)}\Big)\\
& \leq N \|Du^{(k)}\|_{L_2(\widehat Q_{2k},\mu)}
+N\sqrt{\lambda} \|u^{(k)}\|_{L_2(\widehat Q_{2k},\mu)} < \infty.
\end{split}
\]
Using this estimate and \eqref{eq7.38}, we infer that $u^{(k)} \in \cH_p^1(\Omega_T)$. The theorem is proved.
\end{proof}

\section{Equations with partially VMO coefficients} \label{mea-sec}

In this section, we give the proofs of Theorem \ref{thm2}, Corollary \ref{main-thrm}, Theorem \ref{thm1.4}, and Corollary \ref{cor1.8}. We begin with the proof of  Theorem \ref{thm2}.

\subsection{Proof of Theorem \ref{thm2}}
We need  the following decomposition result for our proof.
\begin{proposition}
                            \label{G-approx-propos}
Let $\gamma_0 \in (0, 1)$, $\alpha \in (-1, \infty)$,  $r \in (0, \infty)$,  $z_0\in \overline{\Omega_T}$, and  $q \in (2,\infty)$. Suppose that $G=  |F| + |f| \in L_{2}(Q_{2r}^+(z_0), \mu)$ and $u \in \cH^{1}_q(Q_{2r}^+(z_0),\mu)$ is a weak solution of \eqref{eq3.23}. If \textup{Assumption \ref{assump1} ($\gamma_0, R_0$)} is satisfied and $\textup{spt}(u) \subset   (s - (R_0r_0)^2, s + (R_0r_0)^2) \times \bR^{d}_+$ for some $r_0>0$ and $s \in \bR$, then we have
\[
u(t, x) = v(t, x) + w(t, x) \quad  \text{in}\  Q_{2r}^+(z_0),
\]
where $v$ and $w$ are functions in $\cH_2^1(Q_{2r}^+(z_0), \mu)$ that satisfy
\begin{align} \nonumber
\fint_{Q_{2r}^+(z_0)} |V|^2 \,{\mu}(dz) &  \leq N \fint_{Q_{2r}^+(z_0)} |G|^{2} \,{\mu}(dz)   \\ \label{B-u-tilde-est-inter}
& \qquad +  N(\gamma_0^{1-2/q}+ r_0^{2-4/q}) \left(\fint_{Q_{2r}^+(z_0)} |Du|^q \,{\mu}(dz) \right)^{2/q}
\end{align}
and
\begin{align} \label{D-L-infty-w-inter}
\|W\|_{L_\infty(Q_{r}^+(z_0))}^{2} \leq N  \fint_{Q_{2r}^+(z_0)} |U|^{2}\,{\mu}(dz) +  N\fint_{Q_{2r}^+(z_0)}  |G|^{2} \,{\mu}(dz),
\end{align}
where
$$
V=|Dv|+ \sqrt{\lambda}|v|,\quad W=|Dw|+\sqrt{\lambda}|w|,\quad U=|D u|+\sqrt{\lambda}|u|,
$$
and
$N = N(d,\alpha, \kappa,  q)$.
\end{proposition}
\begin{proof}
For $i=1,2,\ldots, d$, let
\[
b_{i}(t,x) = \chi_{Q_{2r}^+(z_0)}(z) \big( a_{ij}(t,x) - {[a_{ij}]_{2r,z_0}}(x_d)\big)D_j u (t,x)  - F_i(z)\chi_{Q_{2r}^+(z_0)}(z),
\]
where $[a_{ij}]_{2r,z_0}(x_d)$ is defined in Assumption \ref{assump1}.
Observe that $b_i \in L_2(\Omega_T, {\mu})$. In particular, if $r \in (0,R_0/2)$,  it follows from H\"{o}lder's inequality  and Assumption \ref{assump1} ($\gamma_0, R_0$)
that
\begin{align*}
 & \fint_{Q_{2r}^+(z_0)} |b(z)|^2 \,{\mu}(dz)  \\
& \leq \left(\fint_{Q_{2r}^+(z_0)} |a_{ij} -[a_{ij}]_{2r,z_0}|^{\frac{2q}{q-2}} \,{\mu}(dz) \right)^{\frac{q-2}{q}} \left(\fint_{Q_{2r}^+(z_0)} |D u|^{q} \,{\mu}(dz) \right)^{\frac{2}{q}} \\ \nonumber
& \qquad +  \fint_{Q_{2r}^+(z_0)} | F |^{2} \,{\mu}(dz) \nonumber\\
& \leq N \gamma_0^{\frac{q-2}{q}} \left(\fint_{Q_{2r}^+(z_0)} | D u |^q \,{\mu}(dz) \right)^{2/q} + \fint_{Q_{2r}^+(z_0)} | F |^{2} \,{\mu}(dz) .
\end{align*}
On the other hand, when $r > R_0/2$, as $\text{spt}(u) \subset  (s - (R_0r_0)^2, s + (R_0r_0)^2) \times \bR^{d}_+$ and by the boundedness of $(a_{ij})$ in \eqref{ellipticity}, we have
\begin{align*} 
& \fint_{Q_{2r}^+(z_0)} |b(z)|^2 \,{\mu}(dz)  \\
& \leq N(\kappa) \left(\fint_{Q_{2r}^+(z_0)}  \chi_{(s - (R_0r_0)^2, s + (R_0r_0))^2}(t)  \,{\mu}(dz) \right)^{\frac{q-2}{q}} \left(\fint_{Q_{2r}^+(z_0)} |D u|^{q} \,{\mu}(dz) \right)^{\frac{2}{q}} \\
& \qquad + N \fint_{Q_{2r}^+(z_0)} | F |^{2} \,{\mu}(dz)  \nonumber\\
& \leq N \Big(\frac{R_0 r_0}{r} \Big)^{\frac{2(q-2)}{q}} \left(\fint_{Q_{2r}^+(z_0)} | D u |^q \,{\mu}(dz) \right)^{2/q} +  N \fint_{Q_{2r}^+(z_0)} | F |^{2} \,{\mu}(dz) \\
& \leq N r_0^{\frac{2(q-2)}{q}} \left(\fint_{Q_{2r}^+(z_0)} | D u |^q \,{\mu}(dz) \right)^{2/q} +  N \fint_{Q_{2r}^+(z_0)} | F |^{2} \,{\mu}(dz) .
\end{align*}
Hence, for every $r \in (0, \infty)$ we have
\begin{equation} \label{G-inter-est}
\begin{split}
 \fint_{Q_{2r}^+(z_0)} |b(z)|^2 \,{\mu}(dz)  & \leq N \Big(r_0^{\frac{2(q-2)}{q}} + \gamma_0^{\frac{q-2}{q}}\Big) 
 \left(\fint_{Q_{2r}^+(z_0)} | D u |^q \,{\mu}(dz) \right)^{2/q} \\
 & \qquad + N \fint_{Q_{2r}^+(z_0)} | F |^{2} \,{\mu}(dz) .
\end{split}
\end{equation}
Now let $v \in \cH_2^1(\Omega_T, \mu)$ be a weak solution in $\Omega_T$ of
\begin{equation*} 
\left\{
\begin{aligned}
x_d^\alpha  (\partial_t v +\lambda v) -  D_i \big(x_d^\alpha([a_{ij}]_{2r,z_0}(x_d) D_{j} v+ b_i) \big)  & = \lambda^{1/2} x_d^\alpha f \chi_{Q_{2r}^+(z_0)}, \\
\lim_{x_d \rightarrow 0^+} x_d^\alpha \big([a_{dj}]_{2r,z_0}(x_d) D_{j} v + b_d\big) & =  0.
\end{aligned}
\right.
\end{equation*}
By Lemma \ref{L-2-lemma} and \eqref{G-inter-est}, we have
\begin{equation} \label{U-2.est}
\begin{split}
& \fint_{Q_{2r}^+(z_0)}|V|^2 \,\mu(dz) \leq N \fint_{Q_{2r}^+(z_0)}\Big( |b|^2 + |f|^2 \Big)\,\mu(dz) \\
& \leq N \big(\gamma_0^{1-2/q} + r_0^{2-4/q}\big) \left(\fint_{Q_{2r}^+(z_0)} | D u |^q \,{\mu}(dz) \right)^{2/q}  + N \fint_{Q_{2r}^+(z_0)} |G|^{2} \,{\mu}(dz),
\end{split}
\end{equation}
which yields \eqref{B-u-tilde-est-inter}. Let $w = u - v \in \cH_2^1(Q_{2r}^+(z_0), \mu)$, which is a weak solution of
\[
x_d^\alpha(w_t + \lambda w) - D_i \big(x_d^\alpha  [a_{ij}]_{2r, z_0} (x_d) D_{j} w\big) =0 \quad \text{in} \  Q_{2r}^+(z_0)
\]
with the boundary condition
\[
\lim_{x_d\rightarrow 0^+} x_d^\alpha  [a_{dj}]_{2r, z_0} (x_d) D_{j} w =0 \quad \text{if} \quad B_{2r}(x_0) \cap \partial \overline{\bR}^{d}_+ \not= \emptyset.
\]
Then we apply Proposition \ref{prop1} to conclude that
\[
\begin{split}
& \| W\|_{L_\infty(Q_r^+(z_0))}  \leq N \left(\fint_{Q_{2r}^+(z_0)} |W|^2 \,{\mu}(dz)\right)^{1/2} \\
& \leq  N \left(\fint_{Q_{2r}^+(z_0)} |U|^2 \,{\mu}(dz)\right)^{1/2} +
N\left(\fint_{Q_{2r}^+(z_0)} |V|^2 \,{\mu}(dz)\right)^{1/2}.
\end{split}
\]
From this and \eqref{U-2.est}, we obtain \eqref{D-L-infty-w-inter}.
\end{proof}
\begin{proof}[Proof of Theorem \ref{thm2}]
It suffices to consider the case $p \in (2, \infty)$ as the case $p \in (1,2)$ can be  proved by  using the duality argument as in the proof of Theorem \ref{thm3}.  We first prove the a-priori estimate \eqref{main-thm-est} for each weak solution $u \in \cH_p^1(\Omega_T, \mu)$ of \eqref{eq3.23}. We suppose that $\lambda>0$.
Assume for a moment that
$$
\textup{spt}(u) \subset  (s - (R_0r_0)^2, s + (R_0r_0)^2) \times \bR^{d}_+
$$
with some $s \in (-\infty, T)$ and $r_0 \in (0,1)$. We claim that  \eqref{main-thm-est} holds if $\gamma_0$ and  $r_0$ are sufficiently small depending on $d$, $\alpha$, $\kappa$, and $p$. Let $q \in (2, p)$ be fixed. Applying Proposition \ref{G-approx-propos}, for each $r >0$ and $z_0 \in \overline{\Omega}_T$, we can write
\[
u(t, x) = v(t, x) + w(t, x) \quad  \text{in}\ Q_{2r}^+(z_0),
\]
where $v$ and $w$ satisfy \eqref{B-u-tilde-est-inter} and \eqref{D-L-infty-w-inter}.  Then it follows from the standard real variable argument,  see \cite{DK11b} for example, that
\[
\begin{split}
& \|Du\|_{L_p(\Omega_T, \mu)} + \sqrt{\lambda} \|u\|_{L_p(\Omega_T, \mu)}  \\
& \leq N(\gamma_0^{1-2/q} + r_0^{2-4/q}) \|Du\|_{L_{p}(\Omega_T, \mu)} + N\|F\|_{L_p(\Omega_T, \mu)} +N\|f\|_{L_p(\Omega_T, \mu)}
\end{split}
\]
for $N = N(d, \alpha, \kappa, p)$. From this, and by choosing $\gamma_0$ and $r_0$ sufficiently small so that $N(\gamma_0^{1-2/q} + r_0^{2-4/q}) < 1/2$, we obtain \eqref{main-thm-est}.

We now remove the additional assumption that $\textup{spt}(u) \subset  (s -( R_0r_0)^2, s+ (R_0r_0)^2) \times \bR^{d}_+$ by using a partition of unity argument. Let
$$
\xi=\xi(t) \in C_0^\infty(-(R_0r_0)^2, (R_0r_0)^2)
$$
be a standard non-negative cut-off function satisfying
\begin{equation} \label{xi-0515}
\int_{\bR} \xi^p(s)\, ds =1, \quad  \int_{\bR}|\xi'(s)|^p\,ds \leq \frac{N}{(R_0r_0)^{2p}}.
\end{equation}
For any $s \in (-\infty,  \infty)$, let $u^{(s)}(z) = u(z) \xi(t-s)$ for $z = (t, x) \in \Omega_T$. Then $u^{(s)} \in \cH_p^1(\Omega_T, \mu)$ is a weak solution of
\[
\left\{
\begin{aligned}
x_d^\alpha( u^{(s)}_t + \lambda u^{(s)}) -D_i\big(x_d^\alpha(a_{ij} D_j u^{(s)} - F^{(s)}_{i})\big) & = \lambda^{1/2} x_d^\alpha f^{(s)}\\
\lim_{x_d \rightarrow 0^+}x_d^\alpha (a_{dj}D_j u^{(s)} -F_d^{(s)} ) & = 0
\end{aligned} \right.
\]
in $\Omega_T$, where
\[
F^{(s)}(z) = \xi(t-s) F(z), \quad f^{(s)}(z)   = \xi(t-s) f(z)  +  \lambda^{-1/2}\xi'(t-s) u(z).
\]
As $\text{spt}(u^{(s)}) \subset (s -( R_0r_0)^2, s+ (R_0r_0)^2) \times \bR^{d}_{+}$, we can apply the estimate we just proved and infer that
\[
 \|Du^{(s)}\|_{L_p(\Omega_T, \mu)} + \sqrt{\lambda} \|u^{(s)}\|_{L_p(\Omega_T, \mu)}  \leq N \|F^{(s)}\|_{L_p(\Omega_T, \mu)} +N\|f^{(s)}\|_{L_p(\Omega_T, \mu)}.
\]
Integrating with respect to $s$, we get
\begin{equation} \label{int-0515}
\begin{split}
& \int_{\bR}\Big( \|Du^{(s)}\|_{L_p(\Omega_T, \mu)}^p + \lambda^{p/2} \|u^{(s)}\|^p_{L_p(\Omega_T, \mu)}\Big)\, ds\\
&  \leq N\int_{\bR} \Big( \|F^{(s)}\|^p_{L_p(\Omega_T, \mu)} + \|f^{(s)}\|^p_{L_p(\Omega_T, \mu)} \Big)\, ds.
\end{split}
\end{equation}
It follows from the Fubini theorem and \eqref{xi-0515} that
\[
\int_{\bR}\|Du^{(s)}\|_{L_p(\Omega_T,\mu)}^p\, ds = \int_{\Omega_T}\int_{\bR} |Du(z)|^p \xi^p(t-s)\, ds\,\mu(dz)  = \|Du\|_{L_p(\Omega_T, \mu)}^p.
\]
Similarly,
\[
 \int_{\bR}\|u^{(s)}\|_{L_p(\Omega_T,\mu)}^p\, ds = \|u\|_{L_p(\Omega_T, \mu)}^p,  \quad \int_{\bR}\|F^{(s)}\|_{L_p(\Omega_T,\mu)}^p\, ds = \|F\|_{L_p(\Omega_T, \mu)}^p.
\]
Since $r_0$ depends only on $d$, $\alpha$, $\kappa$, and $p$, from the definition of $f^{(s)}$, \eqref{xi-0515}, and the Fubini theorem, we have
\[
\begin{split}
\left(\int_{\bR} \|f^{(s)}\|_{L_p(\Omega, \mu)}^p\, ds \right)^{1/p} \leq N  \|f\|_{L_p(\Omega_T, \mu)} +  NR_0^{-2} \lambda^{1/2}\|u\|_{L_p(\Omega_T, \mu)}
\end{split}
\]
for $N = N(d,\alpha, \kappa, p)$. Collecting these estimates, we infer from \eqref{int-0515} that
\[
\begin{split}
& \|Du\|_{L_p(\Omega_T, \mu)} + \sqrt{\lambda} \|u\|_{L_p(\Omega_T, \mu)}\\
&  \leq N\|F \|_{L_p(\Omega_T, \mu)} +N\|f\| _{L_p(\Omega_T, \mu)} + NR_0^{-2}\lambda^{-1/2}\|u\|_{L_p(\Omega_T, \mu)}
\end{split}
\]
with $N=N(d,\alpha, \kappa, p)$. Now we choose $\lambda_0 = 2N$. For $\lambda \geq \lambda_0 R_0^{-2}$, we have $NR_0^{-2}\lambda^{-1/2} \leq \sqrt\lambda/2$, and therefore
\[
\begin{split}
& \|Du \|_{L_p(\Omega_T, \mu)}
+ \sqrt{\lambda} \|u \|_{L_p(\Omega_T, \mu)}\\
&  \leq   N \|F \|_{L_p(\Omega_T, \mu)} +N \|f\| _{L_p(\Omega_T, \mu)}+\frac{\sqrt{\lambda}}{2} \|u\|_{L_p(\Omega_T, \mu)} ,
\end{split}
\]
which yields \eqref{eq3.23}.

Finally, the solvability of solution $u \in \cH_p^1(\Omega_T, \mu)$ can be obtained by the method of continuity using the solvability of the equation
\[
\left\{
\begin{aligned}
x_d(u_t + \lambda u) - D_i(x_d^\alpha D_i u - F_i) & = \lambda^{1/2} x_d^\alpha f \\
\lim_{x_d \rightarrow 0^+}x_d^\alpha(D_d u - F_d ) & =  0
\end{aligned} \right.
\]
in $\Omega_T$,  which is proved in Theorem \ref{thm3}. The proof is now completed.
\end{proof}
\subsection{Proof of Corollary \ref{main-thrm}} We adapt an idea in \cite{KRW20}.  Let $\eta\in C_0^\infty((-4,4)\times B_2)$ be such that $\eta\equiv 1$ on $Q_1$. A direct calculation yields that $u\eta\in \cH^1_{p_0}(\Omega_0,\mu)$ satisfies
\begin{equation} \label{eq6.16}
\left\{
\begin{aligned}
x_d^\alpha ((u\eta)_t+\lambda u\eta)- D_i\big(x_d^\alpha (a_{ij} D_j (u\eta) - \widetilde F_i)\big)   &=  x_d^\alpha \tilde f\\
\lim_{x_d \rightarrow 0^+}x_d^\alpha \big(a_{dj} D_j (u\eta) - \widetilde F_d\big)  &= 0
\end{aligned} \right.
\ \ \text{in}\ \ (-4,0)\times \bR^d_+
\end{equation}
with the zero initial condition $(u\eta)(-4,\cdot)=0$,
where
$$
\widetilde F_i=F_i\eta-a_{ij}u D_j\eta,
\quad \tilde f=f\eta+\lambda u\eta +u\eta_t-a_{ij}D_i\eta (D_ju-F_i),
$$
and $\lambda> \lambda_0 R_0^{-2}$.

Let $q=p/(p-1)$, and $G=(G_1,\ldots,G_d),g\in C_0^\infty(Q_1^+)$ satisfying
$$
\|G\|_{L_q(Q_1^+,\mu)}=\|g\|_{L_q(Q_1^+,\mu)}=1.
$$
By Theorem \ref{thm2}, there is a weak solution
 $v\in \cH^1_{q}((-4,0)\times \bR^d_+,\mu)$ to
\begin{equation} \label{eq6.25}
\left\{
\begin{aligned}
-x_d^\alpha (v_t -\lambda v) - D_i\big(x_d^\alpha({a}_{ji} D_j v - G_i)\big) & =  \sqrt\lambda x_d^\alpha g\\
\lim_{x_d\rightarrow 0^+}\big(x_d^\alpha({a}_{jd} D_j v - G_d)\big)  &= 0
\end{aligned} \right.
\ \ \text{in}\ \ (-4,0)\times \bR^d_+
\end{equation}
with the zero terminal condition $v(0,\cdot)=0$ and it satisfies
\begin{equation}
                                        \label{eq2.53}
\sqrt \lambda \|v\|_{L_q((-4,0)\times \bR^d_+,\mu)}+\|Dv\|_{L_q((-4,0)\times \bR^d_+,\mu)}\le N.
\end{equation}
Testing \eqref{eq6.16} and \eqref{eq6.25} with $v$ and $u\eta$ respectively, we get
\begin{equation*}
\int_{Q_1^+}(\nabla u\cdot G+\sqrt \lambda u g)\,d\mu(z)=\int_{Q_2^+}(\nabla v\cdot \widetilde F+v \tilde f)\,d\mu(z),
\end{equation*}
which together with H\"older's inequality gives
\begin{align}
                                \label{eq7.12}
&\Big|\int_{Q_1^+}(\nabla u\cdot G+\sqrt \lambda u g)\,d\mu(z)\Big|\notag\\
&\le \|Dv\|_{L_q(Q_2^+,\mu)}\|\widetilde F\|_{L_p(Q_2^+,\mu)}+\|v\|_{L_{q^*}(Q_2^+,\mu)}\|\tilde f\|_{L_{p^*}(Q_2^+,\mu)},
\end{align}
where $q^*=p^*/(p^*-1)$. From \eqref{eq6.25}, we see that $v\in \cH^1_{q}(Q_{2}^+,\mu)$ satisfies
\begin{equation*} 
\left\{
\begin{aligned}
-x_d^\alpha v_t - D_i\big(x_d^\alpha({a}_{ji} D_j v - G_i)\big)  &=  x_d^\alpha \tilde g\\
\lim_{x_d\rightarrow 0^+}x_d^\alpha({a}_{jd} D_j v - G_d)  &= 0
\end{aligned} \right.
\quad\text{in}\ \ Q_{2},
\end{equation*}
where $\tilde g= -\lambda v + \sqrt{\lambda}g$. When $\alpha\neq 0$, by \eqref{eq3.19}-\eqref{eq3-2.19}, $q^*$ satisfies the condition \eqref{eq2.06} in Lemma \ref{lem2.2}. Then by using Lemma \ref{lem2.2} and \eqref{eq2.53}, we get
\begin{align}
                                \label{eq3.07}
\|v\|_{L_{q^*}(Q_2^+,\mu)}&\le N\|v\|_{L_{q}(Q_{2}^+,\mu)}+ N\|Dv\|_{L_{q}(Q_{2}^+,\mu)} +  N\|v_t\|_{\cH_q^{-1}(Q_{2}^+, \mu)}   \notag\\
&\leq N+ N\|G\|_{L_{q}(Q_{2}^+,\mu)}+N\|\tilde g\|_{L_{q}(Q_{2}^+,\mu)}
\le N\sqrt\lambda.
\end{align}
When $\alpha=0$, by the usual unweighted parabolic Sobolev embedding, we still get \eqref{eq3.07}. It then follows from \eqref{eq7.12}, \eqref{eq2.53}, \eqref{eq3.07}, and the arbitrariness of $G$ and $g$ that
\begin{align}
                    \label{eq2.57}
&\|Du\|_{L_p(Q_1^+,\mu)}+\sqrt \lambda \|u\|_{L_p(Q_1^+,\mu)}\notag\\
&\le N\|\widetilde F|_{L_p(Q_2^+,\mu)}+N \sqrt\lambda \|\tilde f\|_{L_{p^*}(Q_2^+,\mu)}\notag\\
&\le N\|F\|_{L_p(Q_2^+,\mu)}+N\|u\|_{L_p(Q_2^+,\mu)}
+N \sqrt\lambda\|f\|_{L_{p^*}(Q_2^+,\mu)}\notag\\
&\quad +N  \sqrt \lambda (\lambda +1)\|u\|_{L_{p^*}(Q_2^+,\mu)}
+N \sqrt\lambda \|Du\|_{L_{p^*}(Q_2^+,\mu)},
\end{align}
where $N$ is independent of $\lambda$.  Recall that $p^*<p$.
Finally, from \eqref{eq2.57} we conclude \eqref{main-thm-estb} by using H\"older's inequality and a standard iteration argument for a sufficiently large $\lambda$.  See, for example, \cite[pp. 80--82]{Giaq}. The corollary is proved.

\subsection{Proof of Theorem \ref{thm1.4}}

It follows from Corollary \ref{main-thrm} and Proposition \ref{prop1} that for any $q_0\in (1,2)$, if $v \in \cH^1_{p_0}(Q^+_r(z_0), \mu)$ is a weak solution of \eqref{eq11.52}-\eqref{eq12.01}, we have
\begin{equation}
                                        \label{eq5.52b}
\begin{split}
& [D_{x'}v]_{C^{1/2,1}(Q_{r/2}^+(z_0))}
+[\mathcal{V}]_{C^{1/2,1}(Q_{r/2}^+(z_0))}
+\sqrt \lambda [v]_{C^{1/2,1}(Q_{r/2}^+(z_0))} \\
& \le Nr^{-1}\Big(\fint_{Q^+_{r}(z_0)}|Dv|^{q_0}
+\lambda^{q_0/2}|v|^{q_0}\,\mu(dz)\Big)^{1/q_0},
\end{split}
\end{equation}
where $\mathcal{V} =\overline{a}_{dj}(x_d) D_jv$. By using \eqref{eq5.52b}, Theorem \ref{thm2}, and a decomposition argument as in the proof of Proposition  \ref{G-approx-propos}, we have the following the mean oscillation estimate:  if $\textup{spt}(u) \subset  (s -( R_0r_0)^2, s+ (R_0r_0)^2) \times \bR^{d}_+$ for some $s\in \bR$, then for any $\tau\le 30$ and $z_0\in \overline{\Omega_T}$,
\begin{align*}
&\fint_{Q_{\tau r}^+(z_0)}|D_{x'}u-(D_{x'}u)_{Q_{r}^+(z_0)}|
+| \cU-(\cU)_{Q_{r}^+(z_0)}|
+\sqrt\lambda|u-(u)_{Q_{r}^+(z_0)}|\,{\mu}(dz)\\
&\leq  N  \tau^{-(d+2+\alpha_+)}r_0^{2(1-\frac 1 {q_0})}
\Big(\fint_{Q_{r}^+(z_0)}|Du|^{q_0} \,{\mu}(dz)\Big)^{\frac 1 {q_0}}\\
&\ +N  \tau^{-\frac {d+2+\alpha_+} {q_0}}\Big(\fint_{Q_{r}^+(z_0)}( |F|^{q_0} +  |\sqrt\lambda f|^{q_0}) \,{\mu}(dz)\Big)^{\frac 1 {q_0}}\\
&\ +N \tau \Big(\fint_{Q^+_{r}(z_0)}|Du|^{q_0}
+\lambda^{q_0/2}|u|^{q_0}\,\mu(dz)\Big)^{\frac 1{q_0}}\\
&\ +N \tau^{-\frac {d+2+\alpha_+} {q_0}}\gamma_0^{\frac 1 {q_0\nu_1}} \Big(\fint_{Q^+_{r}(z_0)}|Du|^{q_0\nu_2}\,\mu(dz)\Big)^{\frac 1 {q_0\nu_2}},
\end{align*}
where $\nu_1\in (1,\infty)$, $\nu_2=\nu_1/(\nu_1-1)$, and $\cU = a_{dj} D_j u$. Here we used the notation
\[
(g)_{Q_r^+(z_0)} = \fint_{Q_r^+(z_0)} g(z) \,\mu(dz)
\]
for a function $g$ defined in $Q_r^+(z_0)$. The a-priori estimate \eqref{main-thm-estc}
then follows from the mean oscillation estimate, the reverse H\"older's inequality for $A_p$ weights, the weighted mixed-norm Fefferman--Stein type theorems on sharp functions, and the weighted mixed-norm Hardy--Littlewood maximal function theorem. See, for instance, Corollary 2.6, 2.7, and  Section 7 of \cite{MR3812104} for details. The solvability in weighted mixed-norm Sobolev spaces then follows from the estimate \eqref{main-thm-estc} and an approximate argument by using the solvability result in Theorem \ref{thm2}. We omit the details and refer the reader to \cite[Section 8]{MR3812104}.

\subsection{Proof of Corollary \ref{cor1.8}}
We first assume that
\begin{equation}
                        \label{eq1.02}
(d+ 3+\alpha_+)/p_0< 1+(d+ 3+\alpha_+)/p.
\end{equation}
Let $\eta\in C_0^\infty((-4,4)\times B_2)$ be an even function with respect to $x_d$ such that $\eta\equiv 1$ on $Q_1$. A direct calculation yields that $w: = u\eta\in W^{1,2}_{p_0}(\Omega_0,\mu)$ satisfies
\begin{equation} \label{eq6.16b}
\left\{
\begin{aligned}
a_0w_t- a_{ij}D_{ij}w - \frac \alpha {x_d} a_{dd}D_d w+\lambda c_0 w   &=  \tilde f\\
\lim_{x_d \rightarrow 0^+}x_d^\alpha a_{dd} D_d w  &= 0
\end{aligned} \right.
\ \ \text{in}\ \ (-4,0)\times \bR^d_+
\end{equation}
with the zero initial condition $w(-4,\cdot)=0$,
where
\[
\begin{split}
\tilde f &=f\eta+\big(a_0 \eta_t-a_{ij}D_{ij}\eta-\alpha a_{dd}D_d\eta/x_d+ (\lambda-1) c_0 \eta\big)u\\
& \qquad   -(a_{ij}+a_{ji})D_i\eta D_j u,\\
\end{split}
\]
$\lambda> \lambda_0 R_0^{-2}$ is a fixed number, and $\lambda_0$ is the constant from Theorem \ref{para.theorem} with $q=p$ and $\omega\equiv K=1$. It follows from Lemma \ref{lem2.2} and \eqref{eq1.02} that
\begin{equation}
                                    \label{eq1.13}
\|u\|_{L_p(Q_2^+,\mu)}+\|Du\|_{L_p(Q_2^+,\mu)}\le N\|u\|_{W^{1,2}_{p_0}(Q_2^+,\mu)}.
\end{equation}
By using Theorem \ref{para.theorem} with $q=p$ and $\omega\equiv K=1$, \eqref{eq6.16b} has a unique solution $v\in W^{1,2}_{p}(\Omega_0,\mu)$. Since $\tilde f$ is compactly supported, as in {\bf Case II} of the proof of Theorem \ref{thm3}, we have $v\in W^{1,2}_{p_0}(\Omega_0,\mu)$. Now by the uniqueness of $W^{1,2}_{p_0}(\Omega_0,\mu)$-solutions to \eqref{eq6.16b}, we conclude that $u\eta=v\in W^{1,2}_{p}(\Omega_0,\mu)$. Furthermore, by Theorem  \ref{para.theorem} and \eqref{eq1.13},
$$
\|u\|_{W^{1,2}_p(Q_1^+,\mu)}\le N\|\tilde f\|_{W^{1,2}_p(Q_2^+,\mu)}
\le N\|f\|_{W^{1,2}_p(Q_2^+,\mu)}+N\|u\|_{W^{1,2}_{p_0}(Q_2^+,\mu)},
$$
which, together with  H\"older's inequality and a standard iteration argument, yields \eqref{eq6.27} under the additional condition \eqref{eq1.02}.

Finally, for general $p\in (p_0,\infty)$, the result follows from an induction argument by taking a sequence of increasing exponents $p_j,j=1,\ldots,n,$ such that $p_n=p$ and
$$
(d+ 3+\alpha_+)/p_{j-1}< 1+(d+ 3+\alpha_+)/p_j
$$
for $j=1,\ldots,n$.

\end{document}